\documentclass{amsart}
\usepackage[english]{babel}
\usepackage{amsmath}
\usepackage{amsthm}
\usepackage{amssymb}
\usepackage[all]{xy}
\usepackage{xkeyval}
\usepackage{graphicx}
\usepackage{makecell}
\usepackage{latexsym}
\usepackage{exscale}
\usepackage{comment}

\usepackage{amscd}
\usepackage{ytableau}

\usepackage{hyperref}

\usepackage[all]{xy}

\usepackage{amsbsy}

\def\frak{\mathfrak}

\newtheorem{prop}[equation]{Proposition}
\newtheorem*{prop*}{Proposition}
\newtheorem{thm}[equation]{Theorem}
\newtheorem*{thm*}{Theorem}
\newtheorem{lem}[equation]{Lemma}
\newtheorem*{lem*}{Lemma}

\newtheorem*{kor*}{Corollary}

\numberwithin{equation}{section}

\newcommand{\frg}{\mathfrak{g}}

\newcommand{\frk}{\mathfrak{k}}

\newcommand{\frp}{\mathfrak{p}}

\newcommand{\frt}{\mathfrak{t}}

\def\bbC{\mathbb{C}}

\def\bbZ{\mathbb{Z}}

\let\ccdot\cdot
\def\cdot{\hbox to 2.5pt{\hss$\ccdot$\hss}}

\newcommand{\p}{{\frak p}}

\newcommand{\eq}{\begin{equation}}
	\newcommand{\eeq}{\end{equation}}
\newcommand{\eqn}{\begin{equation*}}
	\newcommand{\bmul}{\begin{multline*}}
		\newcommand{\eemul}{\end{multline*}}
	\newcommand{\eeqn}{\end{equation*}}
\newcommand{\pf}{\begin{proof}}
	\newcommand{\epf}{\end{proof}}

\newcommand{\la}{\lambda}

\renewcommand{\phi}{\varphi}

\newcommand{\eps}{\varepsilon}

\let\ssize\scriptstyle
\newif\ifFIRST\newdimen\MAXright\MAXright0pt
\def\sdynkin{\bgroup\eightpoint\dynkin}
\def\endsdynkin{\enddynkin\egroup}
\def\dynkin{\bgroup\FIRSTtrue\hskip.5em\setbox1\hbox{$\diagup$}%
	\setbox2\hbox{$\diagdown$}%
	\setbox0\hbox to2\wd1{\hrulefill}%
	\setbox3\hbox{$\bullet$}%
	\setbox4\hbox{$\times$}%
	\setbox7\hbox{$\circ$}
	\def\whiteroot##1{\ifFIRST\setbox5\hbox{$##1$}\ifdim\wd5>1.3em
		\hskip-.5em\hskip.5\wd5\fi\fi\FIRSTfalse
		\hskip-.25em\raise1.5\wd3\hbox to0pt{\hss\hskip.45em$
			\ssize##1$\hss}\copy7\hskip-.25em\setbox6\hbox{$##1$}
		\MAXright\wd6}
	\def\root##1{\ifFIRST\setbox5\hbox{$##1$}\ifdim\wd5>1.3em%
		\hskip-.5em\hskip.5\wd5\fi\fi\FIRSTfalse%
		\hskip-.25em\raise1.5\wd3\hbox to0pt{\hss\hskip.45em$%
			\ssize##1$\hss}\copy3\hskip-.25em\setbox6\hbox{$##1$}%
		\MAXright\wd6}%
	\def\whitedroot##1{\ifFIRST\setbox5\hbox{$##1$}\ifdim\wd5>1.3em
		\hskip-.5em\hskip.5\wd5\fi\fi\FIRSTfalse
		\hskip-.25em\lower1.8\wd3\hbox to0pt{\hss\hskip.45em$
			\ssize##1$\hss}\copy7\hskip-.25em\setbox6\hbox{$##1$}
		\MAXright\wd6}%
	\def\whiterroot##1{\hskip-.25em\copy7\hbox to0pt{\hskip.3em$\ssize##1$\hss}%
		\hskip-.25em\setbox6\hbox{\hskip.6em$##1##1$}%
		\MAXright\wd6}%
	\def\droot##1{\ifFIRST\setbox5\hbox{$##1$}\ifdim\wd5>1.3em%
		\hskip-.5em\hskip.5\wd5\fi\fi\FIRSTfalse%
		\hskip-.25em\lower1.8\wd3\hbox to0pt{\hss\hskip.45em$%
			\ssize##1$\hss}\copy3\hskip-.25em\setbox6\hbox{$##1$}%
		\MAXright\wd6}%
	\def\rroot##1{\hskip-.25em\copy3\hbox to0pt{\hskip.3em$\ssize##1$\hss}%
		\hskip-.25em\setbox6\hbox{\hskip.6em$##1##1$}%
		\MAXright\wd6}%
	\def\norroot##1{\hskip-.36em\copy4\hbox to0pt{\hskip.3em$\ssize##1$\hss}%
		\hskip-.48em\setbox6\hbox{\hskip.6em$##1##1$}%
		\MAXright\wd6}%
	\def\noroot##1{\ifFIRST\setbox5\hbox{$##1$}\ifdim\wd5>1.3em%
		\hskip-.5em\hskip.5\wd5\fi\fi\FIRSTfalse%
		\hskip-.36em\raise1.5\wd3\hbox to0pt{\hss\hskip.6em$%
			\ssize##1$\hss}\copy4\hskip-.38em\setbox6\hbox{$##1$}%
		\MAXright\wd6}%
	\def\nodroot##1{\ifFIRST\setbox5\hbox{$##1$}\ifdim\wd5>1.3em%
		\hskip-.5em\hskip.5\wd5\fi\fi\FIRSTfalse%
		\hskip-.36em\lower1.8\wd3\hbox to0pt{\hss\hskip.6em$%
			\ssize##1$\hss}\copy4\hskip-.38em\setbox6\hbox{$##1$}%
		\MAXright\wd6}%
	\def\nolink{\hskip\wd0}
	\def\link{\raise.22em\copy0}%
	\def\llink##1{\raise.32em\copy0\hskip-\wd0%
		\raise.12em\copy0\hskip-.5\wd0\hbox to0pt{\hss$##1$\hss}\hskip.5\wd0}%
	\def\lllink##1{\raise.22em\copy0\hskip-\wd0\raise.32em\copy0\hskip-\wd0%
		\raise.12em\copy0\hskip-.5\wd0\hbox to0pt{\hss$##1$\hss}\hskip.5\wd0}%
	\def\rootupright##1{\hbox to0pt{\raise.45em\copy1\hskip-.25em\raise1.3\ht1%
			\hbox{\copy3\hskip.3em$\ssize##1$}\hss}%
		\setbox6\hbox{\hskip.6em\copy1\copy1$##1##1$}%
		\ifdim\MAXright<\wd6\MAXright\wd6\fi}%
	\def\whiterootupright##1{\hbox to0pt{\raise.45em\copy1\hskip-.25em\raise1.3\ht1
			\hbox{\copy7\hskip.3em$\ssize##1$}\hss}
		\setbox6\hbox{\hskip.6em\copy1\copy1$##1##1$}
		\ifdim\MAXright<\wd6\MAXright\wd6\fi}
	\def\norootupright##1{\hbox to0pt{\raise.45em\copy1\hskip-.36em\raise1.3\ht1%
			\hbox{\copy4\hskip.3em$\ssize##1$}\hss}%
		\setbox6\hbox{\hskip.6em\copy1\copy1$##1##1$}%
		\ifdim\MAXright<\wd6\MAXright\wd6\fi}%
	\def\rootdownright##1{\hbox to0pt{\raise-.5em\copy2\hskip-.25em\raise-1.35\ht1%
			\hbox{\copy3\hskip.3em$\ssize##1$}\hss}\setbox6%
		\hbox{\hskip.6em\copy2\copy2$##1##1$}%
		\ifdim\MAXright<\wd6\MAXright\wd6\fi}%
	\def\whiterootdownright##1{\hbox to0pt{\raise-.5em\copy2\hskip-.25em\raise-1.35\ht1
			\hbox{\copy7\hskip.3em$\ssize##1$}\hss}\setbox6
		\hbox{\hskip.6em\copy2\copy2$##1##1$}
		\ifdim\MAXright<\wd6\MAXright\wd6\fi}
	\def\rootdown##1{\hbox to0pt{\hskip-.05em\vrule height.25em depth.65em%
			\hskip-.25em\raise-.95em\hbox{\copy3\hskip.3em$\ssize##1$}\hss}%
		\setbox6\hbox{$##1$}%
		\ifdim\MAXright<\wd6\MAXright\wd6\fi}%
	\def\whiterootdown##1{\hbox to0pt{\hskip-.05em\vrule height.25em depth.65em
			\hskip-.25em\raise-.95em\hbox{\copy7\hskip.3em$\ssize##1$}\hss}
		\setbox6\hbox{$##1$}
		\ifdim\MAXright<\wd6\MAXright\wd6\fi}
	\def\dots{\hskip.5em\cdots\hskip.5em}}%
\def\enddynkin{\ifdim\MAXright>1em\hskip.5\MAXright\else\hskip.5em\fi\egroup}%

\begin{document} 
	
	\title[Dirac inequality]{Dirac inequality for highest weight Harish-Chandra modules II}
	\author{Pavle Pand\v zi\'c}
	\address[Pand\v zi\'c]{Department of Mathematics, Faculty of Science, University of Zagreb, Bijeni\v cka 30, 10000 Zagreb, Croatia}
	\email{pandzic@math.hr}
	\author{Ana Prli\'c}
	\address[Prli\'c]{Department of Mathematics, Faculty of Science, University of Zagreb, Bijeni\v cka 30, 10000 Zagreb, Croatia}
	\email{anaprlic@math.hr}
	\author{Vladim\'{\i}r Sou\v cek}
	\address[Sou\v cek]{Matematick\'y \'ustav UK, Sokolovsk\'a 83, 186 75 Praha 8, Czech Republic}
	\email{soucek@karlin.mff.cuni.cz}
	\author{V\'it Tu\v cek}
	\address[Tu\v cek]{Department of Mathematics, Faculty of Science, University of Zagreb, Bijeni\v cka 30, 10000 Zagreb, Croatia}
	\email{}
	\date{}
	\thanks{P.~Pand\v zi\'c, A.~Prli\'c are and V.~Tu\v cek were supported by the QuantiXLie  Center of Excellence, a project 
		cofinanced by the Croatian Government and European Union through the European Regional Development Fund - the Competitiveness and Cohesion Operational Programme 
		(KK.01.1.1.01.0004). V.~Sou\v cek is
		supported by the grant GACR GX19-28628.}
	\begin{abstract} 
		Let $G$ be a connected simply connected noncompact exceptional simple Lie group of Hermitian type. In this paper, we work with the Dirac inequality which is a very useful tool for the classification of unitary highest weight modules. 
	\end{abstract}

	\maketitle
	
	\section{Introduction}
	
	\bigskip
	
	Let $G$ be a connected simply connected noncompact exceptional simple Lie group of Hermitian type. That means that $G$ is either of type $E_6$ or of type $E_7$. Let $\Theta$ be a Cartan involution of $G$ and let $K$ be the group of fixed points of $\Theta$. Then $K/Z$ is a maximal compact subgroup of $G/Z$, where $Z$ denotes the center of $G$.
	
	We will denote by $\mathfrak{g}_0$ the Lie algebra of $G$ and by $\mathfrak{k}_0$  the Lie algebra of  $K$. Let $\mathfrak{g}_0 = \mathfrak{k}_0 \oplus \mathfrak{p}_0$ be the Cartan decomposition and let $\mathfrak{t}_0$ be a Cartan subalgebra of $\mathfrak{k}_0$. Our assumptions on $G$ imply that $\mathfrak{t}_0$ is also a Cartan subalgebra of $\mathfrak{g}_0$. We delete the subscript $0$ to denote complexifications of Lie algebras.
	
	Let $\Delta^{+}_{\frg}\supset\Delta^{+}_{\frk}$ denote fixed sets of positive respectively positive compact roots. Since the pair $(G, K)$ is Hermitian, we have a $K$--invariant decomposition $\frp = \frp^{+} \oplus \frp^{-}$ and $\frp^ {\pm}$ are abelian subalgebras of $\frp$. Let $\rho$ denote the half sum of positive roots  for $\frg$.

	We will consider $\lambda \in \mathfrak{t}^{*}$ which are $\Delta^{+}_{\frk}$-- dominant integral ($\frac{2 \left < \lambda, \alpha \right >}{\left < \alpha, \alpha \right >} \in \mathbb{N} \cup \{0\}$). Let $N(\lambda)$ denote the generalized Verma module. By definition $N(\lambda)= S(\frp^{-}) \otimes F_{\lambda}$, where  $F_{\lambda}$ is the irreducible $\frk$--module with highest weight $\lambda$. The generalized Verma module  $N(\lambda)$ is a highest weight module.  In case $N(\lambda)$ is not irreducible, we will consider the irreducible quotient $L(\lambda)$ of $N(\lambda)$. Our main goal is to  determine those weights $\lambda$ which correspond to unitarizable $L(\lambda)$ using the Dirac inequality. We consider only real highest weights $\lambda$ since this is a necessary condition for unitarity.
	
	To learn more about highest weight modules see \cite{A}, \cite{DES}, \cite{EHW}, \cite{EJ}, \cite {ES}, \cite{J}.

	The $K$--types of $S(\frp^ {-})$ are called the Schmid modules. For each of the  Lie algebras in Table \ref{tab: table2}, the general Schmid module $s$ is a nonnegative integer combination of the so called basic Schmid modules. The basic Schmid modules for each exceptional Lie algebra $\frg_0$ for which $(G, K)$ is a Hermitian symmetric pair are given in Table \ref{tab: table2}. To learn more about the Schmid modules see \cite{S}.
	
	The Dirac operator is an element of $U(\frg) \otimes C(\mathfrak{p})$ defined as $D = \sum_i b_i \otimes d_i$ where $b_i$ is a basis of $\mathfrak{p}$ and $d_i$ is the dual basis of $\mathfrak{p}$ with respect to the Killing form $B$. The Dirac operator acts on the tensor product $X \otimes S$ where $X$ is a $(\frg, K)$--module, and $S$ is the spin module for $C(\p)$. The square of the Dirac operator is:
	\[
	D^2 =  -(\text{Cas}_{\mathfrak{g}} \otimes 1 + \|\rho\|^2) + (\text{Cas}_{\mathfrak{k}_{\Delta}} + \|\rho_{\frk}^2\|),
	\]
	where $\rho_\frk$ is a half sum of the compact positive roots. To learn more about the Dirac operators in representation theory see \cite{H}, \cite{HP1}, \cite{HP2}, \cite{HKP}).
	
	\bigskip
	
	If a $(\frg, K)$--module is unitary, then $D$ is  self adjoint with respect to an inner product, so $D^2 \geq 0$. By the formula for $D^2$ the Dirac inequality becomes explicit on any $K$--type $F_{\tau}$ of $L(\lambda) \otimes S$ 
	$$
	\|\tau + \rho_{\frk} \|^2 \geq \|\lambda + \rho\|^2.
	$$

	In \cite{EHW} it was proved that $L(\lambda)$ is unitary if and only if $D^2 > 0$ on $F_\mu\otimes\bigwedge^{\rm top}\frp^+$ for any $K$--type $F_{\mu}$ of $L(\lambda)$ other than $F_{\lambda}$, that is if and only if 
	$$
	\| \mu + \rho \|^2 > \| \lambda + \rho \|^2.
	$$
	
	The following theorem gives us motivation to study the Dirac inequality (see \cite{PPST1}):
	
	\begin{thm}
		\label{cor unit nonunit}
		Let us assume that $\frg, \rho, \lambda, s$ are as in  tables \ref{tab:table1} and \ref{tab: table2}.
		(1) Let $s_0$ be a Schmid module such that 
		the strict Dirac inequality
		\eq
		\label{strict di s}
		\|(\la-s)^++\rho\|^2> \|\la+\rho\|^2
		\eeq
		holds for any Schmid module $s$ of strictly lower level than $s_0$, and such that
		\[
		\|(\la-s_0)^++\rho\|^2< \|\la+\rho\|^2.
		\]
		Then $L(\la)$ is not unitary.
		
		(2)	If 
		\eq
		\label{strict di s}
		\|(\la-s)^++\rho\|^2> \|\la+\rho\|^2
		\eeq
		holds for all Schmid modules $s$, then $N(\la)$ is irreducible and unitary.
	\end{thm}
	
	In Theorem \ref{cor unit nonunit}, $(\lambda - s)^{+}$ is the unique $\frk$-dominant $W_{\frk}$-conjugate of $\lambda - s$, which means that $(\lambda - s)^{+}$ is as in the third column of Table \ref{tab: table2}.
	
	The proof of the above theorem requires some tools from representation theory, so we will omit it in this paper and prove it in \cite{PPST2}.
	
	\bigskip
	
	\begin{table}[h]
		\begin{center}
			\caption{$\rho$ and $W_{\frk}$}
			\label{tab:table1}
			\begin{tabular}{ |c|c|c|}
				\hline
				Lie algebra & $\rho$ & generators of $W_{\mathfrak{k}}$  \\ 
				\hline \hline
				$\mathfrak{e}_6$ & $(0, 1, 2, 3, 4, -4, -4, 4)$ & $s_{\eps_i \pm \eps_j}, \ 5 \geq i > j$
				\\
				\hline
				$\mathfrak{e}_7$ & $\left(0, 1, 2, 3, 4, 5, - \frac{17}{2}, \frac{17}{2} \right) $ & \makecell{$s_{\eps_i \pm \eps_j}, \ 5 \geq i > j$, \\ $s_{\frac{1}{2}(\eps_8 - \eps_7 -  \eps_6 - \eps_5 - \eps_4 - \eps_3 - \eps_2 + \eps_1)}$} \\
				\hline
			\end{tabular}
		\end{center}
	\end{table}
	
	\bigskip

	\begin{table}[h]
		\begin{center}
			\caption{The weights of basic Schmid modules and the condition for the $\frk$-highest weights $\lambda = (\lambda_1, \lambda_2, \ldots, \lambda_n)$}
			\label{tab: table2}
			\begin{tabular}{ |c|c|c| }
				\hline
				Lie algebra  & basic Schmid modules & highest weights \\ 
				\hline \hline
				
				$\mathfrak{e}_6$  & \makecell{
					$s_1  = \frac{1}{2} \left( 1, 1, 1, 1, 1, -1, -1, 1 \right)$, \\ 
					$s_2  = \left( 0, 0, 0, 0, 1, -1, -1, 1 \right)$} & \makecell{$\lambda = (\lambda_1, \lambda_2, \la_3, \la_4, \la_5, \la_6, \la_6, -\lambda_6)$ \\ $|\la_1|\leq\la_2\leq\cdots\leq\la_5$ \\ $\la_i - \la_j \in \bbZ, \ 2\la_i \in \bbZ, \ i, j\leq 5.$}
				\\
				\hline
				$\mathfrak{e}_7$ & \makecell{
					$s_1 =  \left( 0, 0, 0, 0, 0, 0, -1, 1 \right),$ \\
					$s_2 = \left( 0, 0, 0, 0, 1, 1, -1, 1 \right),$ \\
					$s_3 = \left( 0, 0, 0, 0, 0, 2, -1, 1 \right)$	
				} & \makecell{$\lambda = (\lambda_1, \lambda_2, \la_3, \la_4, \la_5, \la_6, \la_7, -\lambda_7)$ \\ $|\la_1|\leq\la_2\leq\cdots\leq\la_5$\\ $\la_i - \la_j \in \bbZ, \ 2\la_i \in \bbZ, \ i, j\leq 5$ \\ and $\frac{1}{2} \left( \lambda_8 - \sum_{i = 2}^{7} \lambda_i + \lambda_1 \right) \in \mathbb{N}_0 $ } \\
				\hline
			\end{tabular}
		\end{center}
	\end{table}

	In Table \ref{tab:table1}, $s_{\alpha}(\lambda) = \lambda - \frac{2 \langle \lambda, \alpha \rangle}{\langle\alpha,\alpha\rangle} \alpha$ is the reflection of $\lambda$ with respect to the hyperplane orthogonal to a root $\alpha$, $W_{\frk}$ is the Weyl group of $\frk$ generated by the $s_{\alpha}$ and $\mathbb{N}_0 = \mathbb{N} \cup \{0\}$.
	
	Here $\lambda$ and $\rho$ are elements of $\frt^*$ which is identified with $\bbC^n$, and $\eps_i$ denotes the projection to the $i$-th coordinate. The roots are certain functionals on $\frt^*$ and the relevant ones are those in the subscripts of the reflections $s$ in Table \ref{tab:table1}, like $\eps_i-\eps_j$ or $\eps_i+\eps_j$.
	
	\bigskip

	We will frequently use the following lemma in our calculations (see \cite{PPST1}): 
	
	\begin{lem}
		\label{gen prv}
		Let $\frg$ be one of the Lie algebras listed in the above tables. Let $\mu$ and $\nu$ be weights as in Table \ref{tab: table2}. Let $w_1,w_2\in W_\frk$. Then
		\[
		\|(w_1\mu-w_2\nu)^++\rho\|^2\geq \|(\mu-\nu)^++\rho\|^2.
		\]
	\end{lem}
	In Lemma \ref{gen prv}, $(w_1\mu-w_2\nu)^+$ is the unique dominant $W_{\frk}$-conjugate of $w_1\mu-w_2\nu$, which means $(w_1\mu-w_2\nu)^+$ is as in the third column of Table \ref{tab: table2}. The proof requires some representation theory and we leave it for \cite{PPST2}.
	
	\medskip

	\section{Dirac inequalities}

	\subsection{Dirac inequality for $\mathfrak{e}_6$} 
	The basic Schmid $\mathfrak{k}$--modules in $S(\p^{-})$ have lowest weight $-s_i$, $i = 1, 2$, where
	\begin{align*}
		s_1 & = \beta_1 =  \frac{1}{2} \left( 1, 1, 1, 1, 1, -1, -1, 1 \right), \\
		s_2 & = \beta_1 + \beta_2 = \left( 0, 0, 0, 0, 1, -1, -1, 1 \right).
	\end{align*}
	The highest weight $(\mathfrak{g} ,K)$--modules have highest weights of the form
	\begin{align*}
		\lambda  = (\lambda_1, \lambda_2, \lambda_3, \lambda_4, \lambda_5, \lambda_6, \lambda_6, - \lambda_6), \quad & |\lambda_1| \leq \lambda_2 \leq \lambda_3 \leq \lambda_4 \leq \lambda_5, \\
		& \lambda_i - \lambda_j \in \mathbb{Z}, \ 2 \lambda_i \in \mathbb{Z}, \quad i, j \in \{1,2,3,4,5\}
	\end{align*}
	In this case
	$$
	\rho = (0, 1 , 2, 3, 4, -4, -4, 4).
	$$
	The basic necessary condition for unitarity is the Dirac inequality
	$$
	||(\lambda - s_1)^{+} + \rho||^2 \geq ||\lambda + \rho||^2.
	$$
	As before, we write $(\lambda - s_1)^{+} = \lambda - \gamma_1$. Then the Dirac inequality is equivalent to 
	$$
	2 \left < \gamma_1 \, | \, \lambda + \rho \right > \leq \|\gamma_1\|^2.
	$$
	
	We have 
	\begin{align*}
		\lambda - s_1 & = \left (\lambda_1 - \frac{1}{2}, \lambda_2 - \frac{1}{2}, \lambda_3 -  \frac{1}{2}, \lambda_4 - \frac{1}{2}, \lambda_5 - \frac{1}{2}, \lambda_6 + \frac{1}{2}, \lambda_6 + \frac{1}{2}, - \lambda_6 - \frac{1}{2} \right ) \\
		\lambda  + \rho & = (\lambda_1, \lambda_2 + 1, \lambda_3 + 2, \lambda_4 + 3, \lambda_5 + 4, \lambda_6 - 4, \lambda_6 - 4, -\lambda_6 + 4)
	\end{align*}
	There are eight cases. 
	
	\textbf{Case 1.1: $\lambda_1 + \lambda_2 \geq 1$.} In this case $\gamma_1 = s_1$. The basic inequality is equivalent to
	$$
	\sum_{i = 1}^{5}\lambda_i  + 20 \leq 3 \lambda_6.
	$$
	\textbf{Case 1.2: $\lambda_2 = - \lambda_1, \lambda_3 - \lambda_2 \geq 1$.}  In this case $\gamma_1 = \frac{1}{2}(-1, -1, 1, 1, 1, -1, -1, 1)$. The basic inequality is equivalent to
	$$
	\sum_{i = 1}^{5}\lambda_i  + 18 \leq 3 \lambda_6.
	$$
	\textbf{Case 1.3: $\lambda_3 = \lambda_2 = - \lambda_1, \lambda_2 > 0,  \lambda_4 - \lambda_2 \geq 1$.}  In this case $\gamma_1 = \frac{1}{2}(-1, 1, -1, 1, 1, -1, -1, 1)$. The basic inequality is equivalent to
	$$
	\sum_{i = 1}^{5}\lambda_i  + 16 \leq 3 \lambda_6.
	$$
	\textbf{Case 1.4: $\lambda_3 = \lambda_2 = \lambda_1 = 0, \lambda_4 \geq 1$.}  In this case $\gamma_1 = \frac{1}{2}(1, -1, -1, 1, 1, -1, -1, 1)$. The basic inequality is equivalent to
	$$
	\sum_{i = 1}^{5}\lambda_i  + 14 \leq 3 \lambda_6.
	$$
	\textbf{Case 1.5: $\lambda_4 = \lambda_3 = \lambda_2 = - \lambda_1, \lambda_2 > 0, \lambda_5 - \lambda_2 \geq 1$.}  In this case $\gamma_1 = \frac{1}{2}(-1, 1, 1, -1, 1, -1, -1, 1)$. The basic inequality is equivalent to
	$$
	\sum_{i = 1}^{5}\lambda_i  + 14 \leq 3 \lambda_6.
	$$
	\textbf{Case 1.6: $\lambda_4 = \lambda_3 = \lambda_2 = \lambda_1 = 0, \lambda_5 - \lambda_2 \geq 1$.}  In this case $\gamma_1 = \frac{1}{2}(-1, -1, -1, -1, 1, -1, -1, 1)$. The basic inequality is equivalent to
	$$
	\sum_{i = 1}^{5}\lambda_i  + 8 \leq 3 \lambda_6.
	$$
	\textbf{Case 1.7: $\lambda_5 = \lambda_4 = \lambda_3 = \lambda_2 = -\lambda_1, \lambda_2 > 0$.}  In this case $\gamma_1 = \frac{1}{2}(-1, 1, 1, 1, -1, -1, -1, 1)$. The basic inequality is equivalent to
	$$
	\sum_{i = 1}^{5}\lambda_i  + 12 \leq 3 \lambda_6.
	$$
	\textbf{Case 1.8: $\lambda_5 = \lambda_4 = \lambda_3 = \lambda_2 = \lambda_1 = 0$.}  In this case $\gamma_1 = \frac{1}{2}(1, -1, -1, -1, -1, -1, -1, 1)$. The basic inequality is equivalent to
	$$
	\sum_{i = 1}^{5}\lambda_i \leq 3 \lambda_6,
	$$
	i.e. $\lambda_6 \geq 0$.
	
	Now we are going to see in which cases the Dirac inequality holds for $s_2$. We have 
	$$
	\lambda - s_2 = (\lambda_1, \lambda_2, \lambda_3, \lambda_4, \lambda_5 - 1, \lambda_6 + 1, \lambda_6 + 1, -\lambda_6 -1).
	$$
	We write $(\lambda - s_2)^{+} = \lambda - \gamma_2$. Then the Dirac inequality for $s_2$
	$$
	\|(\lambda - s_2)^{+} + \rho \|^2 \geq \|\lambda + \rho \|^2
	$$
	is equivalent to
	$$
	2\left < \gamma_2, \lambda + \rho \right > \leq \| \gamma_2 \|^2
	$$
	There are seven cases. 
	
	\textbf{Case 2.1: $\lambda_5 \neq \lambda_4$.} In this case $\gamma_2 = s_2$. The Dirac inequality for $s_2$ is equivalent to
	$$
	\lambda_5  + 14 \leq 3 \lambda_6.
	$$
	\textbf{Case 2.2: $\lambda_5 = \lambda_4 > \lambda_3$.}  In this case $\gamma_2 = (0, 0, 0, 1, 0, -1, -1, 1)$. The Dirac inequality for $s_2$ is equivalent to
	$$
	\lambda_5  + 13 \leq 3 \lambda_6.
	$$
	\textbf{Case 2.3: $\lambda_5 = \lambda_4 = \lambda_3 > \lambda_2$.}  In this case $\gamma_2 = (0, 0, 1, 0, 0, -1, -1, 1)$. The Dirac inequality for $s_2$ is equivalent to
	$$
	\lambda_5  + 12 \leq 3 \lambda_6.
	$$
	\textbf{Case 2.4: $\lambda_5 = \lambda_4 = \lambda_3 = \lambda_2 > | \lambda_1 |$.}  In this case $\gamma_2 = (0, 1, 0, 0, 0, -1, -1, 1)$. The Dirac inequality for $s_2$ is equivalent to
	$$
	\lambda_5  + 11 \leq 3 \lambda_6.
	$$
	\textbf{Case 2.5: $\lambda_5 = \lambda_4 = \lambda_3 = \lambda_2 = \lambda_1 > 0$.}  In this case $\gamma_2 = (1, 0, 0, 0, 0, -1, -1, 1)$. The basic inequality for $s_2$ is equivalent to
	$$
	\lambda_5  + 10 \leq 3 \lambda_6.
	$$
	\textbf{Case 2.6: $\lambda_5 = \lambda_4 = \lambda_3 = \lambda_2 = -\lambda_1 > 0$.}   In this case $\gamma_2 = (-1, 0, 0, 0, 0, -1, -1, 1)$. The Dirac inequality for $s_2$ is equivalent to
	$$
	\lambda_5  + 10 \leq 3 \lambda_6.
	$$
	\textbf{Case 2.7: $\lambda_5 = \lambda_4 = \lambda_3 = \lambda_2 = \lambda_1 = 0$.}  In this case $\gamma_2 = (0,0, 0, 0, -1, -1, -1, 1)$. The Dirac inequality for $s_2$ is equivalent to
	$$
	\lambda_5 + 6 \leq 3 \lambda_6,
	$$
	i.e. $\lambda_6 \geq 2$.
	
	It is easy to see that in the cases $1.1, 1.2, 1.3, 1.4, 1.5$ and $1.7$ if the Dirac inequality holds for $s_1$ then it also holds for $s_2$, since
	$$
	\lambda_5 \leq \sum_{i = 1}^{5} \lambda_i
	$$
	
	Therefore we have three basic cases:
	\bigskip
	
	\textbf{Case 1:} $\lambda_ i = 0, \quad i \in \{1, 2, 3, 4, 5\}$. 
	
	In this case the basic Dirac inequality can be written as
	$$
	\lambda_6 \geq 0.
	$$
	The Dirac inequality for the second basic Schmid module is equivalent to
	$$
	\lambda_6 \geq 2.
	$$
	
	\textbf{Case 2:} $\lambda_ i = 0, \quad i \in \{1, 2, 3, 4\}, \quad  \lambda_5 \neq 0$.
	
	In this case the basic Dirac inequality can be written as
	$$
	\lambda_5 + 8 \leq 3 \lambda_6.
	$$
	The Dirac inequality for the second basic Schmid module is equivalent to
	$$
	\lambda_5 + 14 \leq 3 \lambda_6.
	$$

	\textbf{Case 3:} $\lambda$ is of type 1.1, 1.2, 1.3, 1.4, 1.5 or 1.7, i.e. $(\lambda_1, \lambda_2, \lambda_3, \lambda_4) \neq (0, 0, 0, 0)$.
	The Dirac inequality for the second basic Schmid module is automatically satisfied if the basic Dirac inequality holds.
	
	\bigskip
	
	Let
	$$
	s_{a, b} = a s_1 + b s_2 = \left( \frac{a}{2}, \frac{a}{2}, \frac{a}{2}, \frac{a}{2}, \frac{a}{2} + b, - \frac{a}{2} - b, - \frac{a}{2} - b, \frac{a}{2} + b \right ), \quad a, b \in \mathbb{N}_{0}, \quad  a + b > 0
	$$
	be a general Schmid module.
	
	\begin{thm}(Case 1)
		\label{first five zero e6}
		Let $\lambda$ be the highest weight of the form $\lambda = (0, 0, 0, 0, 0, \lambda_6, \lambda_6, - \lambda_6).$
		\begin{enumerate}
			\item   If $\lambda_6 > 2$ then $\lambda$ satisfies the strict Dirac inequality 
			$$
			\|(\lambda - s_{a,b})^{+} + \rho  \|^2 > \| \lambda + \rho\|^2 \quad  \forall a,b \in \mathbb{N}_{0}, a + b \neq 0.
			$$
			\item  If $0 < \lambda_6 < 2$ then 
			$$
			\|(\lambda - s_2)^{+} + \rho  \|^2 < \| \lambda + \rho\|^2
			$$
			and the strict Dirac inequality holds for any Schmid module of strictly lower level than $s_2$.
			\item If $\lambda_6 < 0$ than the basic Dirac inequality fails. 
		\end{enumerate}
	\end{thm}
	\begin{proof}
		\begin{enumerate}
			\item We have
			$$
			\lambda - s_{a, b} = \left( -\frac{a}{2}, -\frac{a}{2}, -\frac{a}{2}, -\frac{a}{2}, -\frac{a}{2} - b, \lambda_6 +  \frac{a}{2} + b, \lambda_6 + \frac{a}{2} + b, - \lambda_6 -\frac{a}{2} - b \right ), 
			$$
			and therefore
			\begin{align*}
				& (\lambda- s_{a,b})^{+} = \left( -\frac{a}{2}, \frac{a}{2}, \frac{a}{2}, \frac{a}{2}, \frac{a}{2} + b,  \lambda_6 + \frac{a}{2} + b,  \lambda_6 + \frac{a}{2} + b, -\lambda_6 -\frac{a}{2} - b \right ) \\
				&  = \lambda -  \left( \frac{a}{2}, -\frac{a}{2}, -\frac{a}{2}, -\frac{a}{2}, -\frac{a}{2} - b,  -\frac{a}{2} - b,  - \frac{a}{2} - b,   \frac{a}{2} + b \right ). 
			\end{align*}
			
			Then the strict Dirac inequality
			$$
			\|(\lambda - s_{a,b})^{+} + \rho  \|^2 > \| \lambda + \rho\|^2
			$$
			is equivalent to 
			$$
			2 \left <\gamma_{a,b} \, | \, \lambda +\rho \right > < ||\gamma_{a,b}||^2,
			$$
			where $\gamma_{a,b} = \left( \frac{a}{2}, -\frac{a}{2}, -\frac{a}{2}, -\frac{a}{2}, -\frac{a}{2} - b, -\frac{a}{2} - b, -\frac{a}{2} - b, \frac{a}{2} + b \right)$
			and this inequality is equivalent to
			$$
			-2 a^2 - 4 b^2 - 4ab - 10a - 8b < 3(\lambda_6 - 4)(a + 2b).
			$$
			Since $\lambda_6 > 2$, $3(\lambda_6 - 4)(a + 2b) > -6(a + 2b)$. Furthermore, the inequality
			$$
			-2 a^2 - 4 b^2 - 4ab - 10a - 8b \leq -6(a + 2b)
			$$
			holds for all $a, b \in \mathbb{N}_{0}, a + b \neq 0$. 
			So the strict Dirac inequality holds for any Schmid module $s_{a, b}$. 
			
			\item If $0 < \lambda_6 < 2$ then 
			$$
			\|(\lambda - s_2)^{+} + \rho  \|^2 < \| \lambda + \rho\|^2.
			$$
			Since the level of $s_2$ is equal to two, and the level of $a s_1 + b s_2$ is equal to $a + 2b$, the only Schmid module of strictly lower level than $s_2$ is $s_1$.
			
			For $s_1$ we have $\lambda_6 > 0$, which implies 
			$$
			\|(\lambda - s_1)^{+} + \rho  \|^2 > \| \lambda + \rho\|^2.
			$$
			
			\item  If $\lambda_6 < 0$ than the basic Dirac inequality obviously fails since the basic Dirac inequality in Case 1 is equivalent to $\lambda_6 \geq 0$.
		\end{enumerate}
	\end{proof}

	\begin{thm}(Case 2)
		
		\label{first four zero e6}
		Let $\lambda$ be the highest weight of the form $\lambda = (0, 0, 0, 0, \lambda_5 , \lambda_6, \lambda_6, - \lambda_6)$ 
		\begin{enumerate}
			\item If  $3\lambda_6 - \lambda_5 > 14$ than $\lambda$ satisfies the strict Dirac inequality 
			$$
			\|(\lambda - s_{a,b})^{+} + \rho  \|^2 > \| \lambda + \rho\|^2 \quad  \forall a,b \in \mathbb{N}_{0}, a + b \neq 0.
			$$
			\item If $8 < 3\lambda_6 - \lambda_5 < 14$ then 
			$$
			\|(\lambda - s_2)^{+} + \rho  \|^2 < \| \lambda + \rho\|^2
			$$
			and the strict Dirac inequality holds for any Schmid module of strictly lower level than $s_2$.
			\item If $3 \lambda_6 - \lambda_5 < 8$ than the basic Dirac inequality fails.
		\end{enumerate}
	\end{thm}
	\begin{proof}
		\begin{enumerate}
			\item We have
			$$
			\lambda - s_{a,b} = \left( -\frac{a}{2}, -\frac{a}{2}, -\frac{a}{2}, -\frac{a}{2}, \lambda_5 -\frac{a}{2} - b, \lambda_6 +  \frac{a}{2} + b, \lambda_6 + \frac{a}{2} + b, - \lambda_6 -\frac{a}{2} - b \right ), 
			$$
			and therefore
			\begin{align*}
				(\lambda - s_{a,b})^{+} & = 
				\begin{cases}
					\left(\frac{a}{2}, \frac{a}{2}, \frac{a}{2}, \frac{a}{2}, \lambda_5 - \frac{a}{2} - b,  \lambda_6 + \frac{a}{2} + b,  \lambda_6 + \frac{a}{2} + b, -\lambda_6 -\frac{a}{2} - b \right ), \ \lambda_5 > a + b \\
					\left(\lambda_5 - \frac{a}{2} - b, \frac{a}{2}, \frac{a}{2}, \frac{a}{2},  \frac{a}{2},  \lambda_6 + \frac{a}{2} + b,  \lambda_6 + \frac{a}{2} + b, -\lambda_6 -\frac{a}{2} - b \right ), \ b \leq \lambda_5 \leq a + b\\
					\left(- \frac{a}{2}, \frac{a}{2}, \frac{a}{2}, \frac{a}{2},  -\lambda_5 + \frac{a}{2} + b,  \lambda_6 + \frac{a}{2} + b,  \lambda_6 + \frac{a}{2} + b, -\lambda_6 -\frac{a}{2} - b \right ), \ \lambda_5 < b
				\end{cases} \\
				&  = 	\begin{cases}
					\lambda  - \left(-\frac{a}{2}, -\frac{a}{2}, -\frac{a}{2}, -\frac{a}{2}, \frac{a}{2} + b,  - \frac{a}{2} - b,  - \frac{a}{2} - b, \frac{a}{2} + b \right), \ \lambda_5 > a + b \\
					\lambda  - \left(-\lambda_5+\frac{a}{2} + b, -\frac{a}{2}, -\frac{a}{2}, -\frac{a}{2}, \lambda_5 - \frac{a}{2},  - \frac{a}{2} - b,  - \frac{a}{2} - b, \frac{a}{2} + b \right ), \ b \leq \lambda_5 \leq a + b\\
					\lambda  - \left(\frac{a}{2}, -\frac{a}{2}, -\frac{a}{2}, -\frac{a}{2}, 2\lambda_5-\frac{a}{2} - b,  - \frac{a}{2} - b,  - \frac{a}{2} - b, \frac{a}{2} + b \right ), \ \lambda_5 < b.
				\end{cases}
			\end{align*}
			Then the strict Dirac inequality
			$$
			\|(\lambda - s_{a,b})^{+} + \rho  \|^2 > \| \lambda + \rho\|^2
			$$
			is equivalent to 
			$$
			\begin{cases}
				2a^2 + 4b^2 + 4ab - 10 a - 32 b + (3\lambda_6 - \lambda_5)(a + 2b) > 0, \ \lambda_5 > a + b \\
				2a^2 + 4b^2 + 4ab - 2 a - 24 b + (3\lambda_6 - \lambda_5)(a + 2b) - 8 \lambda_5 > 0, \ b \leq \lambda_5 \leq a + b \\
				2a^2 + 4b^2 + 4ab - 2 a - 16 b + (3\lambda_6 - \lambda_5)(a + 2b) - 16 \lambda_5 > 0, \ \lambda_5 < b
			\end{cases}
			$$
			Since $3\lambda_6 - \lambda_5 > 14$, then $(3\lambda_6 - \lambda_5)(a + 2b) > 14a + 28b$. To prove the strict Dirac inequality it is enough to prove
			$$
			\begin{cases}
				a^2 + 2b^2 + 2ab + 2a - 2b \geq 0, \ \lambda_5 > a + b \\
				a^2 + 2b^2 + 2ab + 2 a - 2 b  \geq 0, \ b \leq \lambda_5 \leq a + b \\
				a^2 + 2b^2 + 2ab + 6a -2 b \geq 0, \ \lambda_5 < b
			\end{cases}.
			$$
			This is true for all $a, b \in \mathbb{N}_{0}, (a, b) \neq (0, 0)$. So the strict Dirac inequality holds for any Schmid module $s_{a,b}$. 
			
			
			\item If $8 < 3\lambda_6 - \lambda_5 < 14$ then 
			$$
			\|(\lambda - s_2)^{+} + \rho  \|^2 < \| \lambda + \rho\|^2.
			$$
			Since $s_1$ is the only Schmid module of strictly lower level than $s_2$, and for $s_1$ we have $ 3\lambda_6 - \lambda_5 > 8$, it follows that 
			$$
			\|(\lambda - s_1)^{+} + \rho  \|^2 > \| \lambda + \rho\|^2.
			$$
			
			\item If $3\lambda_6 - \lambda_5 < 8$ than the basic Dirac inequality obviously fails since in Case 2 the basic Dirac inequality is equivalent to $3 \lambda_6 - \lambda_5 \geq 8$.
		\end{enumerate}
	\end{proof}
	
	\begin{lem}
		\label{second Schmid e6}
		Let $\lambda$ be a highest weight such that $(\lambda_1, \lambda_2 ,\lambda_3, \lambda_4, \lambda_5) \neq (0, 0, 0, 0, 0)$ and
		$$
		\|(\lambda - s_2)^{+} + \rho \|^2 > \| \lambda + \rho \|^2.
		$$
		Then 
		$$
		\|(\lambda' - s_2)^{+} + \rho \|^2 > \| \lambda' + \rho \|^2,
		$$
		where $\lambda' = (\lambda - s_2)^{+}$. If $\lambda_i' = 0$ for $i = 1, 2, 3, 4 ,5$, then 
		$$
		\|(\lambda' - s_{a, b})^{+} + \rho \|^2 > \| \lambda' + \rho \|^2, \quad \forall a,b \in \mathbb{N}_{0}, \quad a + b \neq 0.
		$$
	\end{lem}
	\begin{proof}
		We have
		$$
		\lambda' = 
		\begin{cases}
			(\lambda_1, \lambda_2, \lambda_3, \lambda_4, \lambda_5 - 1, \lambda_6 + 1, \lambda_6 + 1, -\lambda_6 - 1), \ \lambda \text{ as in case 2.1} \\
			(\lambda_1, \lambda_2, \lambda_3, \lambda_5 - 1, \lambda_5 , \lambda_6 + 1, \lambda_6 + 1, -\lambda_6 - 1), \ \lambda \text{ as in case 2.2} \\
			(\lambda_1, \lambda_2, \lambda_5 - 1, \lambda_5, \lambda_5, \lambda_6 + 1, \lambda_6 + 1, -\lambda_6 - 1), \ \lambda \text{ as in case 2.3} \\
			(\lambda_1, \lambda_5 - 1, \lambda_5, \lambda_5, \lambda_5, \lambda_6 + 1, \lambda_6 + 1, -\lambda_6 - 1), \ \lambda \text{ as in case 2.4} \\
			(\lambda_5 - 1, \lambda_5, \lambda_5, \lambda_5, \lambda_5, \lambda_6 + 1, \lambda_6 + 1, -\lambda_6 - 1), \ \lambda \text{ as in case 2.5} \\
			(-\lambda_5 + 1, \lambda_5, \lambda_5, \lambda_5, \lambda_5, \lambda_6 + 1, \lambda_6 + 1, -\lambda_6 - 1), \ \lambda \text{ as in case 2.6}
		\end{cases}
		$$
		If $\lambda'$ is as in  case 2.1 ($\lambda_5' \neq \lambda_4'$), then $\lambda$ is either as in case 2.1 or as in case 2.2. We have 
		$$\lambda_5' - 3 \lambda_6' = 
		\begin{cases}
			\lambda_5 - 3 \lambda_6 - 4 < -14- 4 = -18, \ \lambda \text{ as in case 2.1} \\ 
			\lambda_5 - 3 \lambda_6 - 3 < -13 - 3 = -16, \ \lambda \text{ as in case 2.2}
		\end{cases}
		$$
		Thus, $\lambda_5' - 3 \lambda_6' < -14$. It follows that the strict Dirac inequality holds for the second basic Schmid module. 
		
		\bigskip
		
		If $\lambda'$ is as in  case 2.2 ($\lambda_5' =  \lambda_4' > \lambda_3'$), then $\lambda$ is either as in case 2.1 or as in case 2.3. We have 
		$$\lambda_5' - 3 \lambda_6' = 
		\begin{cases}
			\lambda_5 - 3 \lambda_6 - 4 < -14- 4 = -18, \ \lambda \text{ as in case 2.1} \\ 
			\lambda_5 - 3 \lambda_6 - 3 < -12 - 3 = -15, \ \lambda \text{ as in case 2.3}
		\end{cases}
		$$
		Thus, $\lambda_5' - 3 \lambda_6' < -13$. It follows that the strict Dirac inequality holds for the second basic Schmid module. 
		
		\bigskip
		
		If $\lambda'$ is as in case 2.3 ($\lambda_5' =  \lambda_4' = \lambda_3' > \lambda_2'$), then $\lambda$ is either as in case 2.1 or as in case 2.4. We have 
		$$\lambda_5' - 3 \lambda_6' = 
		\begin{cases}
			\lambda_5 - 3 \lambda_6 - 4 < -14- 4 = -18, \ \lambda \text{ as in case 2.1} \\ 
			\lambda_5 - 3 \lambda_6 - 3 < -11 - 3 = -14, \ \lambda \text{ as in case 2.4}
		\end{cases}
		$$
		Thus, $\lambda_5' - 3 \lambda_6' < -12$. It follows that the strict Dirac inequality holds for the second basic Schmid module. 
		
		\bigskip
		
		If $\lambda'$ is as in  case 2.4 ($\lambda_5' =  \lambda_4' = \lambda_3' = \lambda_2' > |\lambda_1'|$), then $\lambda$ is either as in case 2.1 or as in case 2.5 or as in case 2.6. We have 
		$$\lambda_5' - 3 \lambda_6' = 
		\begin{cases}
			\lambda_5 - 3 \lambda_6 - 4 < -14- 4 = -18, \ \lambda \text{ as in case 2.1} \\ 
			\lambda_5 - 3 \lambda_6 - 3 < -10 - 3 = -13, \ \lambda \text{ as in case 2.5 or as in case 2.6} 
		\end{cases}
		$$
		Thus, $\lambda_5' - 3 \lambda_6' < -11$. It follows that the strict Dirac inequality holds for the second basic Schmid module. 
		
		\bigskip
		
		If $\lambda'$ is as in  case 2.5 or as in case 2.6 ($\lambda_5' =  \lambda_4' = \lambda_3' = \lambda_2' = |\lambda_1'| > 0$), then $\lambda$ is either as in case 2.1 or as in case 2.5 (for $\lambda_1 = \frac{1}{2}$) or as in case 2.6 (for $\lambda_1 = \frac{1}{2}$). We have 
		$$\lambda_5' - 3 \lambda_6' = 
		\begin{cases}
			\lambda_5 - 3 \lambda_6 - 4 < -14- 4 = -18, \ \lambda \text{ as in case 2.1} \\ 
			\lambda_5 - 3 \lambda_6 - 3 < -10 - 3 = -13, \ \lambda \text{ as in case 2.5 or as in case 2.6}
		\end{cases}
		$$
		Thus, $\lambda_5' - 3 \lambda_6' < -10$. It follows that the strict Dirac inequality holds for the second basic Schmid module. 
		
		\bigskip
		
		If $\lambda'$ is as in  case 2.7 ($\lambda_5' =  \lambda_4' = \lambda_3' = \lambda_2' = \lambda_1'= 0$), then $\lambda = (0, 0, 0, 0, 1, \lambda_6, \lambda_6, - \lambda_6)$ and $1 - 3 \lambda_6 < -14$, that is $\lambda_6 > 5$ and $\lambda_6' = \lambda_6 + 1 > 6 > 2$. The strict Dirac inequality holds for the second basic Schmid module. 
		
		It follows from theorem \ref{first five zero e6} that
		$$
		\|(\lambda' - s_{a,b})^{+} + \rho\|^2 - \|\lambda' + \rho \|^2 > 0 \quad \forall a,b \in \mathbb{N}_{0} \quad a + b \neq 0.
		$$ 
	\end{proof}
	
	\begin{lem}
		\label{first Schmid e6}
		Let $\lambda$ be a highest weight such that $(\lambda_1, \lambda_2 ,\lambda_3, \lambda_4) \neq (0, 0, 0, 0)$ and
		$$
		\|(\lambda - s_1)^{+} + \rho \|^2 > \| \lambda + \rho \|^2.
		$$
		Then 
		$$
		\|(\lambda' - s_1)^{+} + \rho \|^2 > \| \lambda' + \rho \|^2,
		$$
		where $\lambda' = (\lambda - s_1)^{+}$. If $\lambda_i' = 0$ for $i = 1, 2, 3, 4$, then 
		$$
		\|(\lambda' - s_{a, b})^{+} + \rho \|^2 > \| \lambda' + \rho \|^2, \quad \forall a, b \in \mathbb{N}_{0}, \quad a + b \neq 0.
		$$
	\end{lem}
	\begin{proof}
		We have
		$$
		\lambda' = 
		\begin{cases}
			(\lambda_1 - \frac{1}{2}, \lambda_2 - \frac{1}{2}, \lambda_3 - \frac{1}{2}, \lambda_4 - \frac{1}{2}, \lambda_5 - \frac{1}{2}, \lambda_6 + \frac{1}{2}, \lambda_6 + \frac{1}{2}, -\lambda_6 - \frac{1}{2}), \ \lambda \text{ as in case 1.1} \\
			(-\lambda_2 + \frac{1}{2}, \lambda_2 + \frac{1}{2}, \lambda_3 - \frac{1}{2}, \lambda_4 - \frac{1}{2}, \lambda_5 - \frac{1}{2}, \lambda_6 + \frac{1}{2}, \lambda_6 + \frac{1}{2}, -\lambda_6 - \frac{1}{2}), \ \lambda \text{ as in case 1.2} \\
			(-\lambda_2 + \frac{1}{2}, \lambda_2 - \frac{1}{2}, \lambda_2 + \frac{1}{2}, \lambda_4 - \frac{1}{2}, \lambda_5 - \frac{1}{2}, \lambda_6 + \frac{1}{2}, \lambda_6 + \frac{1}{2}, -\lambda_6 - \frac{1}{2}), \ \lambda \text{ as in case 1.3} \\
			( - \frac{1}{2}, \frac{1}{2},  \frac{1}{2}, \lambda_4 - \frac{1}{2}, \lambda_5 - \frac{1}{2}, \lambda_6 + \frac{1}{2}, \lambda_6 + \frac{1}{2}, -\lambda_6 - \frac{1}{2}), \ \lambda \text{ as in case 1.4} \\
			(-\lambda_2 + \frac{1}{2}, \lambda_2 - \frac{1}{2}, \lambda_2 - \frac{1}{2}, \lambda_2 + \frac{1}{2}, \lambda_5 - \frac{1}{2}, \lambda_6 + \frac{1}{2}, \lambda_6 + \frac{1}{2}, -\lambda_6 - \frac{1}{2}), \ \lambda \text{ as in case 1.5} \\
			(-\lambda_2 + \frac{1}{2}, \lambda_2 - \frac{1}{2}, \lambda_2 - \frac{1}{2}, \lambda_2 - \frac{1}{2}, \lambda_2 + \frac{1}{2}, \lambda_6 + \frac{1}{2}, \lambda_6 + \frac{1}{2}, -\lambda_6 - \frac{1}{2}), \ \lambda \text{ as in case 1.7} \\
		\end{cases}
		$$
		If $\lambda'$ is as in  case 1.1 ($\lambda_1' + \lambda_2' \geq 1$), then $\lambda$ is either as in case 1.1 or as in case 1.2. We have 
		$$\sum_{i = 1}^{5} \lambda_i' - 3 \lambda_6' = 
		\begin{cases}
			\sum_{i = 1}^{5} \lambda_i - 3 \lambda_6 - 4 < -20 - 4 = - 24, \ \lambda \text{ as in case 1.1} \\ 
			\sum_{i = 1}^{5} \lambda_i - 3 \lambda_6 - 2 < -18 - 2 = - 20, \ \lambda \text{ as in case 1.2}
		\end{cases}
		$$
		Thus, $\sum_{i = 1}^{5} \lambda_i' - 3 \lambda_6' < - 20$. It follows that the strict basic Dirac inequality holds. 
		
		\bigskip
		
		If $\lambda'$ is as in case 1.2 ($- \lambda_1' = \lambda_2', \lambda_3' - \lambda_2' \geq  1$), then $\lambda$ is either as in case 1.1 or as in case 1.3. We have 
		$$\sum_{i = 1}^{5} \lambda_i' - 3 \lambda_6' = 
		\begin{cases}
			\sum_{i = 1}^{5} \lambda_i - 3 \lambda_6 - 4 < -20 - 4 = - 24, \ \lambda \text{ as in case 1.1} \\ 
			\sum_{i = 1}^{5} \lambda_i - 3 \lambda_6 - 2 < -16 - 2 = - 18, \ \lambda \text{ as in case 1.3}
		\end{cases}
		$$
		Thus, $\sum_{i = 1}^{5} \lambda_i' - 3 \lambda_6' < - 18$. It follows that the strict basic Dirac inequality holds. 
		
		\bigskip
		
		If $\lambda'$ is as  in case 1.3 ($ \lambda_3' = \lambda_2' = - \lambda_1', \lambda_2' > 0, \lambda_4' - \lambda_2' \geq 1$), then $\lambda$ is either as in case 1.1 or as in case 1.4 or as in case 1.5. We have 
		$$\sum_{i = 1}^{5} \lambda_i' - 3 \lambda_6' = 
		\begin{cases}
			\sum_{i = 1}^{5} \lambda_i - 3 \lambda_6 - 4 < -20 - 4 = - 24, \ \lambda \text{ as in case 1.1} \\ 
			\sum_{i = 1}^{5} \lambda_i - 3 \lambda_6 - 2 < -14 - 2 = - 16, \ \lambda \text{ as in case 1.4} \\ 
			\sum_{i = 1}^{5} \lambda_i - 3 \lambda_6 - 2 < -14 - 2 = - 16, \ \lambda \text{ as in case 1.5}
		\end{cases}
		$$
		Thus, $\sum_{i = 1}^{5} \lambda_i' - 3 \lambda_6' < - 16$. It follows that the strict basic Dirac inequality holds.	
		
		\bigskip
		
		If $\lambda'$ is as in case 1.4 ($ \lambda_1' = \lambda_2' =  \lambda_3' = 0, \lambda_4' > 0$), then $\lambda$ is either as in case 1.1 or as in case 1.5. We have 
		$$\sum_{i = 1}^{5} \lambda_i' - 3 \lambda_6' = 
		\begin{cases}
			\sum_{i = 1}^{5} \lambda_i - 3 \lambda_6 - 4 < -20 - 4 = - 24, \ \lambda \text{ as in case 1.1} \\ 
			\sum_{i = 1}^{5} \lambda_i - 3 \lambda_6 - 2 < -14 - 2 = - 16, \ \lambda \text{ as in case 1.5} 
		\end{cases}
		$$
		Thus, $\sum_{i = 1}^{5} \lambda_i' - 3 \lambda_6' < - 14$. It follows that the strict basic Dirac inequality holds.
		
		\bigskip
		
		If $\lambda'$ is as in case 1.5 ($ \lambda_4' = \lambda_3' =  \lambda_2' = - \lambda_1', \lambda_2' > 0, \lambda_5' - \lambda_2' \geq 1$), then $\lambda$ is either as in case 1.1 or as in case 1.4 or as in case 1.7. We have 
		$$\sum_{i = 1}^{5} \lambda_i' - 3 \lambda_6' = 
		\begin{cases}
			\sum_{i = 1}^{5} \lambda_i - 3 \lambda_6 - 4 < -20 - 4 = - 24, \ \lambda \text{ as in case 1.1} \\ 
			\sum_{i = 1}^{5} \lambda_i - 3 \lambda_6 - 2 < -14 - 2 = - 16, \ \lambda \text{ as in case 1.4} \\
			\sum_{i = 1}^{5} \lambda_i - 3 \lambda_6 - 2 < -12 - 2 = - 14, \ \lambda \text{ as in case 1.7}
		\end{cases}
		$$
		Thus, $\sum_{i = 1}^{5} \lambda_i' - 3 \lambda_6' < - 14$. It follows that the strict basic Dirac inequality holds.	
		
		\bigskip
		
		If $\lambda'$ is as in case 1.6 ($ \lambda_4' = \lambda_3' =  \lambda_2' =  \lambda_1' = 0, \lambda_5' - \lambda_2' \geq 1$), then $\lambda$ is either as in case 1.1 or as in case 1.7. We have 
		$$\lambda_5' - 3 \lambda_6' = \sum_{i=1}^{5} \lambda_i' - 3 \lambda_6' = 
		\begin{cases}
			\sum_{i=1}^{5} \lambda_i - 3 \lambda_6 - 4 < -20 - 4 = - 24, \ \lambda \text{ as in case 1.1} \\ 
			\sum_{i = 1}^{5} \lambda_i - 3 \lambda_6 - 2 < -12 - 2 = - 14, \ \lambda \text{ as in case 1.7}
		\end{cases}
		$$
		Thus, $\lambda_5' - 3 \lambda_6' < - 14$. It follows that the strict Dirac inequality for the second basic Schmid module holds and thus, from the proof of theorem \ref{first four zero e6} we have
		$$
		\|(\lambda' - s_{a,b})^{+} + \rho\|^2 - \|\lambda' + \rho \|^2 > 0 \quad \forall a, b \in \mathbb{N}_{0}, \quad  a + b \neq 0.
		$$ 
		
		\bigskip
		
		If $\lambda'$ is as in case 1.7 ($ \lambda_5' = \lambda_4' = \lambda_3' =  \lambda_2' =  -\lambda_1', \lambda_2' > 0$), then $\lambda$ is either as in case 1.1 or as in case 1.4. We have 
		$$\sum_{i=1}^{5} \lambda_i' - 3 \lambda_6' = 
		\begin{cases}
			\sum_{i=1}^{5} \lambda_i - 3 \lambda_6 - 4 < -20 - 4 = - 24, \ \lambda \text{ as in case 1.1} \\ 
			\sum_{i = 1}^{5} \lambda_i - 3 \lambda_6 - 2 < -14 - 2 = - 16, \ \lambda \text{ as in case 1.4}
		\end{cases}
		$$
		Thus, $\sum_{i=1}^{5} \lambda_i' - 3 \lambda_6' < -12$. It follows that the strict basic Dirac inequality holds.
		
		\bigskip
		
		If $\lambda'$ is as in case 1.8 ($ \lambda_5' = \lambda_4' = \lambda_3' =  \lambda_2' =  \lambda_1' = 0$), then $\lambda$ is as in case 1.1. We have 
		$$- 3 \lambda_6' = \sum_{i=1}^{5} \lambda_i' - 3 \lambda_6' = \sum_{i=1}^{5} \lambda_i - 3 \lambda_6 - 4 < -20 - 4 = - 24 
		$$
		
		Thus, $\lambda_6' > 8 > 2$. The strict Dirac inequality holds for the second basic Schmid module. 
		
		It follows from theorem \ref{first five zero e6} that
		$$
		\|(\lambda' - s_{a,b})^{+} + \rho\|^2 - \|\lambda' + \rho \|^2 > 0 \quad \forall a, b \in \mathbb{N}_{0} \quad a + b \neq 0.
		$$ 
	\end{proof}

	\begin{thm} (Case 3) Let $\lambda$ be the highest weight as in Case 3, i.e., $(\lambda_1, \lambda_2, \lambda_3, \lambda_4) \neq (0, 0, 0, 0)$ such that strict basic Dirac inequality holds. Then 
		$$
		\|(\lambda - s_{a,b})^{+} + \rho\|^2 - \| \lambda + \rho\|^2 > 0 \quad  \forall a, b \in \mathbb{N}_{0}, (a, b) \neq (0,0).
		$$
	\end{thm}
	
	\begin{proof}
		Let $\lambda$ be as in Case $3$, and let us assume that the strict basic Dirac inequality holds. First we will prove that in this case we have
		\begin{equation}\label{b}
			\|(\lambda - s_{0,b})^{+} + \rho\|^2 - \|\lambda + \rho \|^2 > 0 \quad \forall b \in \mathbb{N}.
		\end{equation}
		Let us denote $\lambda' = (\lambda - s_2)^{+}$. We have already proved that if $\lambda$ is in Case 3 and the strict basic Dirac inequality holds, then the strict Dirac inequality also holds for $s_2$. So we have 
		$$
		\| \lambda' + \rho\|^2 > \| \lambda + \rho\|^2.
		$$
		Let us assume that $b > 1$. Let $w\in W_{\mathfrak{k}}$ be such that $\lambda - s_2 = w(\lambda - s_2)^{+}$.  From Lemma \ref{gen prv} we have
		\begin{align*}
			\|(\lambda - s_{0,b})^{+} + \rho\|^2 & = \|(\lambda - s_2  - s_{0,b-1})^{+} + \rho\|^2 
			= \| (w(\lambda - s_2)^{+} -  s_{0, b-1})^{+} + \rho \|^2 \\
			& \geq \|((\lambda - s_2)^{+} - s_{0,b-1})^{+} + \rho\|^2 =  \|(\lambda' - s_{0,b-1})^{+} + \rho\|^2.
		\end{align*}
		It follows from the last two inequalities that
		\begin{equation}\label{e6b}
			\|(\lambda - s_{0,b})^{+} + \rho\|^2 - \| \lambda + \rho\|^2 > \|(\lambda' - s_{0,b-1})^{+} + \rho\|^2 - \| \lambda' + \rho\|^2 \quad \forall b > 1
		\end{equation}
		If $\lambda_i' = 0$ for $i = 1, 2, 3, 4 ,5$, then
		it follows from lemma \ref{second Schmid e6} that
		$$
		\|(\lambda' - s_{0, b-1})^{+} + \rho \|^2 > \| \lambda' + \rho \|^2, \quad \forall b > 1,
		$$
		and it follows from \eqref{e6b} that
		$$
		\|(\lambda - s_{0, b})^{+} + \rho \|^2 - \| \lambda + \rho \|^2 > 0, \quad \forall b > 1.
		$$
		Since $\|\lambda' + \rho\|^2 > \| \lambda + \rho \|^2$, we have
		$$
		\|(\lambda - s_{0, b})^{+} + \rho \|^2 - \| \lambda + \rho \|^2 > 0, \quad \forall b \in \mathbb{N}.
		$$
		If $(\lambda_1', \lambda_2', \lambda_3', \lambda_4', \lambda_5' ) \neq (0,0,0,0,0)$ and if $b > 2$ then it follows from lemma \ref{second Schmid e6} and from \eqref{e6b} that
		$$
		\|(\lambda' - s_{0,b-1})^{+} + \rho\|^2 - \| \lambda' + \rho\|^2 > \|(\lambda'' - s_{0,b-2})^{+} + \rho\|^2 - \| \lambda'' + \rho\|^2, 
		$$
		where $\lambda'' = (\lambda' - s_2)^{+}$. By induction, it follows
		$$
		\|(\lambda - s_{0,b})^{+} + \rho\|^2 - \| \lambda + \rho\|^2 > 0 \quad  \forall b \in \mathbb{N}.
		$$
		
		Now we will prove that if $\lambda$ is as in Case 3, and the strict basic Dirac inequality holds, then 
		$$
		\|(\lambda - s_{a,b})^{+} + \rho\|^2 - \|\lambda + \rho \|^2 > 0 \quad \forall a, b \in \mathbb{N}_{0}, (a, b) \neq (0, 0).
		$$ 
		Let us denote $\tilde{\lambda} = (\lambda - s_1)^{+}$. We have 
		$$
		\| \tilde{\lambda} + \rho\|^2 > \| \lambda + \rho\|^2.
		$$
		Let us assume that $a > 1$ or $a = 1, \ b > 0$. Let $\tilde{w}\in W_{\mathfrak{k}}$ be such that $\lambda - s_1 = \tilde{w}(\lambda - s_1)^{+}$.  It follows from Lemma \ref{gen prv} that
		\begin{align*}
			\|(\lambda - s_{a,b})^{+} + \rho\|^2 & = \|(\lambda - s_1  - s_{a - 1, b})^{+} + \rho\|^2 
			= \| (\tilde{w}(\lambda - s_1)^{+} -  s_{a - 1, b})^{+} + \rho \|^2 \\
			& \geq \|((\lambda - s_1)^{+} - s_{a-1,b})^{+} + \rho\|^2 =  \|(\tilde{\lambda} - s_{a - 1,b})^{+} + \rho\|^2.
		\end{align*}
		It follows from the last two inequalities that
		\begin{equation}\label{e6ab}
			\|(\lambda - s_{a,b})^{+} + \rho\|^2 - \| \lambda + \rho\|^2 > \|(\tilde{\lambda} - s_{a - 1,b})^{+} + \rho\|^2 - \| \tilde{\lambda} + \rho\|^2.
		\end{equation}
		If $\tilde{\lambda}_i = 0$ for $i = 1, 2, 3, 4$, then it follows
		from lemma \ref{first Schmid e6} that
		$$
		\|(\tilde{\lambda} - s_{a - 1, b})^{+} + \rho \|^2 > \| \tilde{\lambda} + \rho \|^2,
		$$
		and it follows from \eqref{e6ab} that
		$$
		\|(\lambda - s_{a, b})^{+} + \rho \|^2 - \| \lambda + \rho \|^2 > 0 \quad \forall a, b \in \mathbb{N}_{0}, \ a + b \neq 0.
		$$
		If $(\tilde{\lambda}_1, \tilde{\lambda}_2, \tilde{\lambda}_3, \tilde{\lambda}_4) \neq (0,0,0,0)$ and $a > 1$, then  it follows from lemma \ref{first Schmid e6} and from \eqref{e6ab} that
		$$
		\|(\tilde{\lambda} - s_{a - 1,b})^{+} + \rho\|^2 - \| \tilde{\lambda} + \rho\|^2 > \|(\bar{\lambda} - s_{a - 2,b})^{+} + \rho\|^2 - \| \bar{\lambda} + \rho\|^2, 
		$$
		where $\bar{\lambda} = (\tilde{\lambda} - s_1)^{+}$. By induction and by \eqref{b}, it follows that
		$$
		\|(\lambda - s_{a,b})^{+} + \rho\|^2 - \| \lambda + \rho\|^2 > 0 \quad  \forall a, b \in \mathbb{N}_{0}, (a, b) \neq (0,0).
		$$
	\end{proof}
	
	\subsection{Dirac inequality for $\mathfrak{e}_7$} 
	The basic Schmid $\mathfrak{k}$--modules in $S(\p^{-})$ have lowest weights $-s_i$, $i = 1, 2, 3$, where
	\begin{align*}
		s_1 & = \beta_1 =  \left( 0, 0, 0, 0, 0, 0, -1, 1 \right), \\
		s_2 & = \beta_1 + \beta_2 = \left( 0, 0, 0, 0, 1, 1, -1, 1 \right), \\
		s_3 & = \beta_1 + \beta_2  + \beta_3= \left( 0, 0, 0, 0, 0, 2, -1, 1 \right).
	\end{align*}
	The highest weight $(\mathfrak{g} ,K)$--modules have highest weight of the form
	\begin{align*}
		\lambda  & = (\lambda_1, \lambda_2, \lambda_3, \lambda_4, \lambda_5, \lambda_6, \lambda_7, - \lambda_7), \quad  |\lambda_1| \leq \lambda_2 \leq \lambda_3 \leq \lambda_4 \leq \lambda_5, \\
		& \lambda_i - \lambda_j \in \mathbb{Z}, \ 2 \lambda_i \in \mathbb{Z}, \ 1 \leq i \leq j \leq 5\\
		& \frac{1}{2} \left( \lambda_8 - \lambda_7 - \lambda_6 + \sum_{i = 1}^{5} (-1)^{n(i)} \lambda_i \right) \in \mathbb{N}_{0}, \quad \sum_{n = 1}^{5} n(i) \text{ even}, 
	\end{align*}
	which can be written more shortly as
	\begin{align*}
		& \lambda  = (\lambda_1, \lambda_2, \lambda_3, \lambda_4, \lambda_5, \lambda_6, \lambda_7, - \lambda_7), \quad  |\lambda_1| \leq \lambda_2 \leq \lambda_3 \leq \lambda_4 \leq \lambda_5, \\
		& \lambda_i - \lambda_j \in \mathbb{Z}, \ 2 \lambda_i \in \mathbb{Z}, \ 1 \leq i \leq j \leq 5\\
		& \frac{1}{2} \left( \lambda_8 - \lambda_7 - \lambda_6 -\lambda_5 -\lambda_4 -\lambda_3 - \lambda_2 + \lambda_1  \right) \in \mathbb{N}_{0}.
	\end{align*}
	In this case
	$$
	\rho = \left( 0, 1 , 2, 3, 4, 5, -\frac{17}{2}, \frac{17}{2} \right ).
	$$
	The basic necessary condition for unitarity is the Dirac inequality
	$$
	||(\lambda - s_1)^{+} + \rho||^2 \geq ||\lambda + \rho||^2.
	$$
	As before, we write $(\lambda - s_1)^{+} = \lambda - \gamma_1$. Then the Dirac inequality is equivalent to 
	$$
	2 \left < \gamma_1 \, , \, \lambda + \rho \right > \leq \|\gamma_1\|^2.
	$$
	
	We have 
	\begin{align*}
		\lambda - s_1 & = \left (\lambda_1, \lambda_2, \lambda_3, \lambda_4, \lambda_5, \lambda_6, \lambda_7 + 1, - \lambda_7 - 1 \right ) \\
		\lambda  + \rho & = \left (\lambda_1, \lambda_2 + 1, \lambda_3 + 2, \lambda_4 + 3, \lambda_5 + 4, \lambda_6 + 5, \lambda_7 - \frac{17}{2}, -\lambda_7 +  \frac{17}{2} \right )
	\end{align*}
	There are two basic cases. 
	
	\bigskip
	
	\textbf{Case 1.1: $\frac{1}{2} \left( \lambda_8 - \lambda_7 - \lambda_6 -\lambda_5 -\lambda_4 -\lambda_3 - \lambda_2 + \lambda_1 \right) \geq 1$.} In this case $\gamma_1 = s_1$. The basic inequality is equivalent to
	$$
	\lambda_7 \geq 8.
	$$
	\textbf{Case 1.2: $\frac{1}{2} \left( \lambda_8 - \lambda_7 - \lambda_6 -\lambda_5 -\lambda_4 -\lambda_3 - \lambda_2 + \lambda_1 \right) = 0$.} We have
	\begin{align*}
		& s_{\alpha_1 } \left (\lambda_1, \lambda_2, \lambda_3, \lambda_4, \lambda_5, \lambda_6, \lambda_7 + 1, - \lambda_7 - 1 \right ) \\
		& = \left (\lambda_1 + \frac{1}{2}, \lambda_2 - \frac{1}{2}, \lambda_3 - \frac{1}{2}, \lambda_4 - \frac{1}{2}, \lambda_5 - \frac{1}{2}, \lambda_6 - \frac{1}{2}, \lambda_7 + \frac{1}{2}, - \lambda_7 - \frac{1}{2} \right ) .
	\end{align*}

	In this case we have eight subcases.
	
	\bigskip
	
	\textbf{Case 1.2.1: $\frac{1}{2} \left( \lambda_8 - \lambda_7 - \lambda_6 -\lambda_5 -\lambda_4 -\lambda_3 - \lambda_2 + \lambda_1 \right) = 0, \ \lambda_1 < \lambda_2 $.}
	
	\bigskip
	
	In this case 
	$$(\lambda - s_1)^{+} = \left (\lambda_1 + \frac{1}{2}, \lambda_2 - \frac{1}{2}, \lambda_3 - \frac{1}{2}, \lambda_4 - \frac{1}{2}, \lambda_5 - \frac{1}{2}, \lambda_6 - \frac{1}{2}, \lambda_7 + \frac{1}{2}, - \lambda_7 - \frac{1}{2} \right )$$ and $\gamma_1 = \frac{1}{2}(-1, 1, 1, 1, 1, 1, -1, 1)$. The basic inequality is equivalent to
	$$
	\lambda_7 \geq \frac{15}{2}.
	$$
	\textbf{Case 1.2.2: $\frac{1}{2} \left( \lambda_8 - \lambda_7 - \lambda_6 -\lambda_5 -\lambda_4 -\lambda_3 - \lambda_2 + \lambda_1 \right) = 0, \ \lambda_1 = \lambda_2 < \lambda_3$.}  
	
	\bigskip
	
	In this case $$(\lambda - s_1)^{+} = \left (\lambda_2 -  \frac{1}{2}, \lambda_2 + \frac{1}{2}, \lambda_3 - \frac{1}{2}, \lambda_4 - \frac{1}{2}, \lambda_5 - \frac{1}{2}, \lambda_6 - \frac{1}{2}, \lambda_7 + \frac{1}{2}, - \lambda_7 - \frac{1}{2} \right )$$ and $\gamma_1 = \frac{1}{2}(1, -1, 1, 1, 1, 1, -1, 1)$. The basic inequality is equivalent to
	$$
	\lambda_7 \geq 7.
	$$
	\textbf{Case 1.2.3: $\frac{1}{2} \left( \lambda_8 - \lambda_7 - \lambda_6 -\lambda_5 -\lambda_4 -\lambda_3 - \lambda_2 + \lambda_1 \right) = 0, \ 0 < \lambda_1 = \lambda_2 = \lambda_3 < \lambda_4$.}  
	
	\bigskip
	
	In this case $$(\lambda - s_1)^{+} = \left (\lambda_3 -  \frac{1}{2}, \lambda_3 - \frac{1}{2}, \lambda_3 + \frac{1}{2}, \lambda_4 - \frac{1}{2}, \lambda_5 - \frac{1}{2}, \lambda_6 - \frac{1}{2}, \lambda_7 + \frac{1}{2}, - \lambda_7 - \frac{1}{2} \right )$$ and $\gamma_1 = \frac{1}{2}(1, 1, -1, 1, 1, 1, -1, 1)$. The basic inequality is equivalent to
	$$
	\lambda_7 \geq \frac{13}{2}.
	$$
	\textbf{Case 1.2.4: $\frac{1}{2} \left( \lambda_8 - \lambda_7 - \lambda_6 -\lambda_5 -\lambda_4 -\lambda_3 - \lambda_2 + \lambda_1 \right) = 0, \ 0 = \lambda_1 = \lambda_2 = \lambda_3 < \lambda_4$.}  
	
	\bigskip
	
	In this case $$(\lambda - s_1)^{+} = \left (\frac{1}{2}, \frac{1}{2}, \frac{1}{2}, \lambda_4 - \frac{1}{2}, \lambda_5 - \frac{1}{2}, \lambda_6 - \frac{1}{2}, \lambda_7 + \frac{1}{2}, - \lambda_7 - \frac{1}{2} \right )$$ and $\gamma_1 = \frac{1}{2}(-1, -1, -1, 1, 1, 1, -1, 1)$. The basic inequality is equivalent to
	$$
	\lambda_7 \geq 6.
	$$
	\textbf{Case 1.2.5: $\frac{1}{2} \left( \lambda_8 - \lambda_7 - \lambda_6 -\lambda_5 -\lambda_4 -\lambda_3 - \lambda_2 + \lambda_1 \right) = 0, \ 0 < \lambda_1 = \lambda_2 = \lambda_3 = \lambda_4 < \lambda_5$.}  
	
	\bigskip
	
	In this case $$(\lambda - s_1)^{+} = \left (\lambda_4 -  \frac{1}{2}, \lambda_4 - \frac{1}{2}, \lambda_4 - \frac{1}{2}, \lambda_4 + \frac{1}{2}, \lambda_5 - \frac{1}{2}, \lambda_6 - \frac{1}{2}, \lambda_7 + \frac{1}{2}, - \lambda_7 - \frac{1}{2} \right )$$ and $\gamma_1 = \frac{1}{2}(1, 1, 1, -1, 1, 1, -1, 1)$. The basic inequality is equivalent to
	$$
	\lambda_7 \geq 6.
	$$
	\textbf{Case 1.2.6: $\frac{1}{2} \left( \lambda_8 - \lambda_7 - \lambda_6 -\lambda_5 -\lambda_4 -\lambda_3 - \lambda_2 + \lambda_1 \right) = 0, \ 0 = \lambda_1 = \lambda_2 = \lambda_3 = \lambda_4 < \lambda_5$.}  
	We have
	\begin{align*}
		s_{\alpha_1}  s_{\eps_2 - \eps_1} s_{\eps_{3} + \eps_4} & \left ( \frac{1}{2}, -\frac{1}{2}, -\frac{1}{2}, -\frac{1}{2}, \lambda_5 - \frac{1}{2}, \lambda_6 - \frac{1}{2}, \lambda_7 + \frac{1}{2}, - \lambda_7 - \frac{1}{2} \right ) \\
		= & \left (0, 0, 0, 0, \lambda_5 - 1, \lambda_6 - 1, \lambda_7, - \lambda_7 \right ) .
	\end{align*}
	
	In this case $$(\lambda - s_1)^{+} = \left (0, 0, 0, 0, \lambda_5 - 1, \lambda_6 - 1, \lambda_7, - \lambda_7 \right )$$ and $\gamma_1 = (0, 0, 0, 0, 1, 1, 0, 0)$. The basic inequality is equivalent to
	$$
	\lambda_7 \geq 4.
	$$
	\textbf{Case 1.2.7: $\frac{1}{2} \left( \lambda_8 - \lambda_7 - \lambda_6 -\lambda_5 -\lambda_4 -\lambda_3 - \lambda_2 + \lambda_1 \right) = 0, \ 0 < \lambda_1 = \lambda_2 = \lambda_3 = \lambda_4 = \lambda_5$.}

	In this case $$(\lambda - s_1)^{+} = \left (\lambda_5 -  \frac{1}{2}, \lambda_5 - \frac{1}{2}, \lambda_5 - \frac{1}{2}, \lambda_5 - \frac{1}{2}, \lambda_5 + \frac{1}{2}, \lambda_6 - \frac{1}{2}, \lambda_7 + \frac{1}{2}, - \lambda_7 - \frac{1}{2} \right )$$ and $\gamma_1 = \frac{1}{2}(1, 1, 1, 1,-1, 1, -1, 1)$. The basic inequality is equivalent to
	$$
	\lambda_7 \geq \frac{11}{2}.
	$$
	\textbf{Case 1.2.8: $\frac{1}{2} \left( \lambda_8 - \lambda_7 - \lambda_6 -\lambda_5 -\lambda_4 -\lambda_3 - \lambda_2 + \lambda_1 \right) = 0, \ 0 = \lambda_1 = \lambda_2 = \lambda_3 = \lambda_4 = \lambda_5$.}  
	We have
	\begin{align*}
		s_{\eps_5 - \eps_1} s_{\alpha_1}  s_{\eps_2 + \eps_3} s_{\eps_{4} + \eps_5} & \left ( \frac{1}{2}, -\frac{1}{2}, -\frac{1}{2}, -\frac{1}{2}, - \frac{1}{2}, \lambda_6 - \frac{1}{2}, \lambda_7 + \frac{1}{2}, - \lambda_7 - \frac{1}{2} \right ) \\
		= & \left (0, 0, 0, 0, 1, \lambda_6 - 1, \lambda_7, - \lambda_7 \right ) .
	\end{align*}
	
	In this case $$(\lambda - s_1)^{+} = \left (0, 0, 0, 0, 1, \lambda_6 - 1, \lambda_7, - \lambda_7 \right )$$ and $\gamma_1 = (0, 0, 0, 0, -1, 1, 0, 0)$. The basic inequality is equivalent to
	$$
	\lambda_7 \geq 0.
	$$
	
	\bigskip
	
	Now we are going to see in which cases the Dirac inequality holds for $s_2$. We have 
	$$
	\lambda - s_2 = (\lambda_1, \lambda_2, \lambda_3, \lambda_4, \lambda_5 - 1, \lambda_6 - 1, \lambda_7 + 1, -\lambda_7 -1).
	$$
	We write $(\lambda - s_2)^{+} = \lambda - \gamma_2$. Then the Dirac inequality for $s_2$
	$$
	\|(\lambda - s_2)^{+} + \rho \|^2 \geq \|\lambda + \rho \|^2
	$$
	is equivalent to
	$$
	2\left < \gamma_2, \lambda + \rho \right > \leq \| \gamma_2 \|^2
	$$
	There are seven cases. 
	
	\textbf{Case 2.1: $\lambda_5 > \lambda_4$.} In this case $\gamma_2 = s_2$. The Dirac inequality for $s_2$ is equivalent to
	$$
	\lambda_5  + \lambda_6 - 2 \lambda_7 + 24 \leq 0.
	$$
	\textbf{Case 2.2: $\lambda_5 = \lambda_4 > \lambda_3$.} In this case $$(\lambda - s_2)^{+} = \left (\lambda_1, \lambda_2, \lambda_3, \lambda_5 - 1, \lambda_5, \lambda_6 - 1, \lambda_7 + 1, - \lambda_7 -1 \right )$$ and $\gamma_2 = (0, 0, 0, 1, 0, 1, -1, 1)$. The Dirac inequality for $s_2$ is equivalent to
	$$
	\lambda_5  + \lambda_6 - 2 \lambda_7 + 23 \leq 0.
	$$
	\textbf{Case 2.3: $\lambda_5 = \lambda_4 = \lambda_3 > \lambda_2$.}  In this case $$(\lambda - s_2)^{+} = \left (\lambda_1, \lambda_2, \lambda_5 - 1, \lambda_5, \lambda_5, \lambda_6 - 1, \lambda_7 + 1, - \lambda_7 -1 \right )$$ and  $\gamma_2 = (0, 0, 1, 0, 0, 1, -1, 1)$. The Dirac inequality for $s_2$ is equivalent to
	$$
	\lambda_5  + \lambda_6 - 2 \lambda_7 + 22 \leq 0.
	$$
	\textbf{Case 2.4: $\lambda_5 = \lambda_4 = \lambda_3 = \lambda_2 > | \lambda_1 |$.}  In this case $$(\lambda - s_2)^{+} = \left (\lambda_1, \lambda_5 - 1, \lambda_5, \lambda_5, \lambda_5, \lambda_6 - 1, \lambda_7 + 1, - \lambda_7 -1 \right )$$ and $\gamma_2 = (0, 1, 0, 0, 0, 1, -1, 1)$. The Dirac inequality for $s_2$ is equivalent to
	$$
	\lambda_5  + \lambda_6 - 2 \lambda_7 + 21 \leq 0.
	$$
	\textbf{Case 2.5: $\lambda_5 = \lambda_4 = \lambda_3 = \lambda_2 = \lambda_1 > 0$.} We have two subcases: 
	
	\textbf{Case 2.5.1: $\lambda_5 = \lambda_4 = \lambda_3 = \lambda_2 = \lambda_1 > 0, \frac{1}{2} \left( \lambda_1 - \lambda_2 - \lambda_3 - \lambda_4 - \lambda_5 - \lambda_6 - \lambda_7 + \lambda_8 \right ) \geq 1$.}
	
	\bigskip
	
	In this case $$(\lambda - s_2)^{+} = \left (\lambda_5 - 1, \lambda_5, \lambda_5, \lambda_5, \lambda_5, \lambda_6 - 1, \lambda_7 + 1, - \lambda_7 -1 \right )$$ and $\gamma_2 = (1, 0, 0, 0, 0, 1, -1, 1)$. The Dirac inequality for $s_2$ is equivalent to
	$$
	\lambda_5  + \lambda_6 - 2 \lambda_7 + 20 \leq 0.
	$$
	\textbf{Case 2.5.2: $\lambda_5 = \lambda_4 = \lambda_3 = \lambda_2 = \lambda_1 > 0, \frac{1}{2} \left( \lambda_1 - \lambda_2 - \lambda_3 - \lambda_4 - \lambda_5 - \lambda_6 - \lambda_7 + \lambda_8 \right ) = 0$.}
	We have 
	\begin{align*}
		&  s_{\alpha_1} \left (\lambda_5 - 1, \lambda_5, \lambda_5, \lambda_5, \lambda_5, \lambda_6 - 1, \lambda_7 + 1, - \lambda_7 -1 \right ) \\
		& = \left (\lambda_5 - \frac{1}{2}, \lambda_5 - \frac{1}{2}, \lambda_5 - \frac{1}{2}, \lambda_5 - \frac{1}{2}, \lambda_5 - \frac{1}{2}, \lambda_6 - \frac{3}{2}, \lambda_7 + \frac{1}{2}, - \lambda_7 - \frac{1}{2} \right ) 
	\end{align*}
	In this case $$(\lambda - s_2)^{+} = \left (\lambda_5 - \frac{1}{2}, \lambda_5 - \frac{1}{2}, \lambda_5 - \frac{1}{2}, \lambda_5 - \frac{1}{2}, \lambda_5 - \frac{1}{2}, \lambda_6 - \frac{3}{2}, \lambda_7 + \frac{1}{2}, - \lambda_7 - \frac{1}{2} \right )$$ and $\gamma_2 = \frac{1}{2}(1, 1, 1, 1, 1, 3, -1, 1)$. The Dirac inequality for $s_2$ is equivalent to
	$$
	\lambda_5  + \lambda_6 - 2 \lambda_7 + 19 \leq 0.
	$$
	\textbf{Case 2.6: $\lambda_5 = \lambda_4 = \lambda_3 = \lambda_2 = -\lambda_1 > 0$.} In this case $$(\lambda - s_2)^{+} = \left (-\lambda_5 + 1, \lambda_5, \lambda_5, \lambda_5, \lambda_5, \lambda_6 - 1, \lambda_7 + 1, - \lambda_7 -1 \right )$$ and $\gamma_2 = (-1, 0, 0, 0, 0, 1, -1, 1)$. The Dirac inequality for $s_2$ is equivalent to
	$$
	\lambda_5  + \lambda_6 - 2 \lambda_7 + 20 \leq 0.
	$$
	\textbf{Case 2.7: $\lambda_5 = \lambda_4 = \lambda_3 = \lambda_2 = \lambda_1 = 0$.}  
	We have two subcases: 
	
	\bigskip
	
	\textbf{Case 2.7.1: $\lambda_5 = \lambda_4 = \lambda_3 = \lambda_2 = \lambda_1 = 0, \frac{1}{2} \left( \lambda_1 - \lambda_2 - \lambda_3 - \lambda_4 - \lambda_5 - \lambda_6 - \lambda_7 + \lambda_8 \right ) \geq 1$.}
	
	In this case $$(\lambda - s_2)^{+} = \left (0, 0, 0, 0, 1, \lambda_6 - 1, \lambda_7 + 1, - \lambda_7 -1 \right )$$ and $\gamma_2 = (0, 0, 0, 0, -1, 1, -1, 1)$. The Dirac inequality for $s_2$ is equivalent to
	$$
	\lambda_5  + \lambda_6 - 2 \lambda_7 + 16 \leq 0.
	$$
	\textbf{Case 2.7.2: $\lambda_5 = \lambda_4 = \lambda_3 = \lambda_2 = \lambda_1 = 0, \frac{1}{2} \left( \lambda_1 - \lambda_2 - \lambda_3 - \lambda_4 - \lambda_5 - \lambda_6 - \lambda_7 + \lambda_8 \right ) = 0$.}
	We have 
	\begin{align*}
		& s_{\eps_5 - \eps_4} s_{\eps_4 + \eps_5}\left (0, 0, 0, 0, -1, \lambda_6 - 1, \lambda_7 + 1, - \lambda_7 -1 \right ) = \left (0, 0, 0, 0, 1, \lambda_6 - 1, \lambda_7 + 1, - \lambda_7 -1 \right ) \\
		& s_{\alpha_1} s_{\eps_2 - \eps_1} s_{\eps_3 + \eps_4} s_{\alpha_1} \left (0, 0, 0, 0, 1, \lambda_6 - 1, \lambda_7 + 1, - \lambda_7 -1 \right ) \\
		& = \left( 0, 0, 0, 0, 0, \lambda_6 - 2, \lambda_7, - \lambda_7 \right).
	\end{align*}
	In this case $$(\lambda - s_2)^{+} = \left( 0, 0, 0, 0, 0, \lambda_6 - 2, \lambda_7, - \lambda_7 \right)$$ and $\gamma_2 = (0, 0, 0, 0, 0, 2, 0, 0)$. The Dirac inequality for $s_2$ is equivalent to
	$$
	\lambda_5  + \lambda_6 - 2 \lambda_7 + 8 \leq 0,
	$$
	i.e. $\lambda_7 \geq 2$.

	Now we are going to see in which cases the Dirac inequality holds for $s_3$. We have 
	$$
	\lambda - s_3 = (\lambda_1, \lambda_2, \lambda_3, \lambda_4, \lambda_5, \lambda_6 - 2, \lambda_7 + 1, -\lambda_7 -1),
	$$
	and therefore $(\lambda - s_3)^{+} = \lambda - s_3$. Then the Dirac inequality for $s_3$
	$$
	\|(\lambda - s_3)^{+} + \rho \|^2 \geq \|\lambda + \rho \|^2
	$$
	is equivalent to
	$$
	2\left < s_3, \lambda + \rho \right > \leq \| s_3 \|^2,
	$$
	i.e.,
	$$
	\lambda_6 - \lambda_7 + 12 \leq 0.
	$$
	It is easy to see that in cases $1.1, 1.2.1, 1.2.2, 1.2.3, 1.2.4, 1.2.5$ or $1.2.7$ if the Dirac inequality holds for $s_1$ then it also holds for $s_2$. Let us assume that the Dirac inequality holds for $s_1$. We have
	$$
	\lambda_5 + \lambda_6 \leq \lambda_1 - \lambda_2 - \lambda_3 - \lambda_4 - 2 \lambda_7 \leq -2 \lambda_7,$$ i.e. $$\lambda_5 + \lambda_6 - 2 \lambda_7 \leq -4 \lambda_7 \leq (-4) \cdot \frac{11}{2} = -22,
	$$ 
	and therefore the Dirac inequality obviously holds for $s_2$ if $\lambda$ is in one of the cases 2.3, 2.4, 2.5, 2.6 or 2.7. If $\lambda$ is in case 2.1 or in case 2.2 and also in one of the cases $1.1, 1.2.1, 1.2.2, 1.2.3, 1.2.4$ or $1.2.5$ (if $\lambda$ is in case 2.1 or 2.2, then $\lambda$ can not be in case $1.2.7$) and the Dirac inequality holds for $s_1$ then $\lambda_7 \geq 6$ and therefore 
	$$
	\lambda_5 + \lambda_6 - 2 \lambda_7 \leq -4 \lambda_7 \leq (-4) \cdot 6 = -24,
	$$
	so the Dirac inequality holds for $s_2$.
	
	\bigskip
	
	Furthermore, in cases $1.1, 1.2.1, 1.2.2, 1.2.3, 1.2.4, 1.2.5$ or $1.2.7$  if the Dirac inequality holds for $s_1$ then it also holds for $s_3$, since $\lambda_6 \leq \lambda_1 - \lambda_2 - \lambda_3 - \lambda_4 -\lambda_5 -  2 \lambda_7 \leq -2 \lambda_7$ and therefore $$\lambda_6 -  \lambda_7 + 12 \leq -3 \lambda_7 + 12 \leq (-3) \cdot \frac{11}{2} + 12 < 0.$$
	
	Therefore, we have three basic cases:
	\bigskip
	
	\textbf{Case 1:} $\lambda_ i = 0, \ i \in \{1, 2, 3, 4, 5\}, \lambda_6 = - 2 \lambda_7$ (case 1.2.8) 
	
	In this case the basic Dirac inequality can be written as
	$$
	\lambda_7 \geq 0.
	$$
	The Dirac inequality for the second basic Schmid module is equivalent to
	$$
	\lambda_7 \geq 2.
	$$
	The Dirac inequality for the third basic Schmid module is equivalent to
	$$
	\lambda_7 \geq 4.
	$$
	It is clear that if the Dirac inequality holds for the third basic Schmid module, then it automatically holds for the first and the second basic Schmid module.
	
	\bigskip
	
	\textbf{Case 2:} $\lambda_ i = 0, \ i \in \{1, 2, 3, 4\}, \  \lambda_5 > 0, \ - \lambda_5 - \lambda_6 - 2 \lambda_7 = 0$ (case 1.2.6)
	
	In this case the basic Dirac inequality can be written as
	$$
	\lambda_7 \geq 4.
	$$
	The Dirac inequality for the second basic Schmid module is equivalent to
	$$
	\lambda_7 \geq 6.
	$$
	The Dirac inequality for the third basic  Schmid module is equivalent to
	$$
	\lambda_6 - \lambda_7 + 12 \leq 0.
	$$
	If the Dirac inequality holds for the second basic Schmid module, then it automatically holds for the first and the third basic Schmid module, since
	$$
	\lambda_6 - \lambda_7 + 12 = - \lambda_5 - 3 \lambda_7 + 12 \leq  - 3 \lambda_7 + 12 \leq -18 + 12 < 0.
	$$
	
	\textbf{Case 3:} $\lambda$ is of type $1.1, 1.2.1, 1.2.2, 1.2.3, 1.2.4, 1.2.5$ or $1.2.7$. The Dirac inequality for the second and the third Schmid module is automatically satisfied if the basic Dirac inequality holds.
	
	Let 
	$$
	s_{a, b, c} = a s_1 + b s_2 + c s_3 = \left( 0, 0, 0, 0, b, b + 2c, -a - b - c, a + b + c\right ), \ a, b, c \in \mathbb{N}_{0}, \  a + b + c > 0
	$$
	be a general Schmid module.
	
	\begin{thm}(Case 1)
		\label{first five zero e7}
		Let $\lambda$ be the highest weight of the form $\lambda = (0, 0, 0, 0, 0, -2\lambda_7, \lambda_7, - \lambda_7)$.
		\begin{enumerate}
			\item If  $\lambda_7 > 4$ then $\lambda$ satisfies the strict Dirac inequality for any Schmid module $s_{a, b, c}$, i.e.
			$$
			\|(\lambda - s_{a,b,c})^{+} + \rho  \|^2 > \| \lambda + \rho\|^2, \ a, b, c \in \mathbb{N}_{0}, \ (a,b,c) \neq (0, 0, 0)
			$$
			\item If $2 < \lambda_7 < 4$ then 
			$$
			\|(\lambda - s_3)^{+} + \rho  \|^2 < \| \lambda + \rho\|^2
			$$
			and the strict Dirac inequality holds for any Schmid module of strictly lower level than $s_3$.
			\item If $0 < \lambda_7 < 2$ then 
			$$
			\|(\lambda - s_2)^{+} + \rho  \|^2 < \| \lambda + \rho\|^2
			$$
			and the strict Dirac inequality holds for any Schmid module of strictly lower level than $s_2$.
			\item If $\lambda_7 < 0$ than the basic Dirac inequality fails.
		\end{enumerate}
	\end{thm}
	
	\begin{proof}
		\begin{enumerate}
			\item We have
			\begin{align*}
				\lambda - s_{a, b, c}  & = \left( 0, 0, 0, 0, -b, -2\lambda_7 - b - 2c, \lambda_7  + a + b + c, -\lambda_7 - a - b - c \right ) \\
				& s_{\eps_5 - \eps_1}s_{\alpha_1} s_{\eps_4 - \eps_1} s_{\eps_4 + \eps_5}s_{\eps_2 + \eps_3}s_{\alpha_1}s_{\eps_5 - \eps_4} s_{\eps_4 + \eps_5}(\lambda - s_{a, b, c}) \\  
				&  = \left(0, 0, 0, 0, a, -2 \lambda_7 - a - 2b - 2c, \lambda_7 + c, -\lambda_7 - c \right)
			\end{align*}
			and therefore
			\begin{align*}
				(\lambda- s_{a,b, c})^{+} & = \left(0, 0, 0, 0, a, -2 \lambda_7 - a - 2b - 2c, \lambda_7 + c, -\lambda_7 - c \right) \\
				&  = \lambda -  \left( 0, 0, 0, 0, -a, a + 2b + 2c, -c, c \right ). 
			\end{align*}
			
			Then the strict Dirac inequality
			$$
			\|(\lambda - s_{a,b,c})^{+} + \rho  \|^2 > \| \lambda + \rho\|^2
			$$
			is equivalent to 
			$$
			2 \left <\gamma_{a,b,c} \, , \, \lambda +\rho \right > < ||\gamma_{a,b,c}||^2,
			$$
			where $\gamma_{a,b,c} = \left( 0, 0, 0, 0, -a, a + 2b + 2c, -c, c \right )$
			and this inequality is equivalent to
			$$
			2(-2 \lambda_7(a + 2b + 3c) + a + 10b + 27c) < a^2 + (a + 2b + 2c)^2 + 2 c^2.
			$$
			Since $\lambda_7 > 4$,  $-2 \lambda_7(a + 2b + 3c) < -8(a + 2b + 3c)$. We see that the inequality
			$$
			2(-8(a + 2b + 3c) + a + 10b + 27c) \leq a^2 + (a + 2b + 2c)^2 + 2 c^2
			$$
			holds for all $a, b, c \in \mathbb{N}_{0}, a + b + c \neq 0$. 
			So the strict Dirac inequality holds for any Schmid module $s_{a, b, c}$.

			
			\item If $2 < \lambda_7 < 4$ then 
			$$
			\|(\lambda - s_3)^{+} + \rho  \|^2 < \| \lambda + \rho\|^2.
			$$
			Since the level of $s_i$ is equal to $i$ where $i \in \{1, 2, 3\}$, and the level of $a s_1 + b s_2 + c s_3$ is equal to $a + 2b + 3c$, the only Schmid modules of strictly lower level than $s_3$ are $s_1, s_2$ and $2 s_1$. For $s_i, i \in \{1, 2\}$, we have $\lambda_7 > 2 > 0$, i.e.
			$$
			\|(\lambda - s_i)^{+} + \rho  \|^2 > \| \lambda + \rho\|^2.
			$$
			We have $(\lambda - 2s_1)^{+} = \lambda - (0, 0, 0, 0, -2, 2, 0, 0)$. Therefore, the strict Dirac inequality  for $2 s_1$ is equivalent to $\lambda_7 > - \frac{1}{2}$, which is true since $2 < \lambda_7 < 4$. 
			
			
			\item If $0 < \lambda_7 < 2$ then 
			$$
			\|(\lambda - s_2)^{+} + \rho  \|^2 < \| \lambda + \rho\|^2.
			$$
			Since the level of $s_2$ is equal to $2$ and the level of $a s_1 + b s_2 + c s_3$ is equal to $a + 2b + 3c$, the only Schmid module of strictly lower level than $s_2$ is $s_1$. For $s_1$ we have $\lambda_7 > 0$, which implies
			$$
			\|(\lambda - s_1)^{+} + \rho  \|^2 > \| \lambda + \rho\|^2.
			$$
			
			\item If $\lambda_7 < 0$ than the basic Dirac inequality obviously fails since in Case 1 the basic Dirac inequality is equivalent to $\lambda_7 \geq 0$.
		\end{enumerate}
	\end{proof}

	\begin{thm}(Case 2)
		\label{first four zero e7}
		Let $\lambda$ be the highest weight of the form $\lambda = (0, 0, 0, 0, \lambda_5 , \lambda_6, \lambda_7, - \lambda_7)$ such that $\lambda_5 > 0$ and $ -\lambda_5 - \lambda_6 - 2 \lambda_7 = 0$.
		\begin{enumerate}
			\item If $\lambda_7 > 6$ than $\lambda$ satisfies the strict Dirac inequality for any Schmid module $s_{a, b, c}$, i.e.
			$$
			\|(\lambda - s_{a,b,c})^{+} + \rho  \|^2 > \| \lambda + \rho\|^2, \ a, b, c \in \mathbb{N}_{0}, \ (a,b,c) \neq (0, 0, 0)
			$$
			\item If $4 < \lambda_7 < 6$ then 
			$$
			\|(\lambda - s_2)^{+} + \rho  \|^2 < \| \lambda + \rho\|^2
			$$
			and the strict Dirac inequality holds strictly for any Schmid module of strictly lower level than $s_2$.
			\item If $\lambda_7 < 4$ than the basic Dirac inequality fails.
		\end{enumerate}
	\end{thm}
	\begin{proof}
		\begin{enumerate}
			\item We have
			\begin{align*}
				& \lambda - s_{a, b, c}  = \left( 0, 0, 0, 0, \lambda_5 -b, \lambda_6 - b - 2c, \lambda_7  + a + b + c, -\lambda_7 - a - b - c \right ) \\
				& s_{\alpha_1} s_{\eps_3 + \eps_4} s_{\eps_2 - \eps_1}s_{\alpha_1}(\lambda - s_{a, b, c})  
				= \left(0, 0, 0, 0, \lambda_5 - a - b, \lambda_6 - a - b - 2c, \lambda_7 + b + c, -\lambda_7 -b - c \right)
			\end{align*}
			and therefore
			\begin{align*}
				(\lambda - s_{a,b, c})^{+} & = 
				\begin{cases}
					\left(0, 0, 0, 0, \lambda_5 - a - b, \lambda_6 - a - b - 2c, \lambda_7 + b + c, -\lambda_7 -b - c \right), \ \lambda_5 > a + b \\
					\left(0, 0, 0, 0, -\lambda_5 + a + b, \lambda_6 - a - b - 2c, \lambda_7 + b + c, -\lambda_7 -b - c \right), \ b \leq \lambda_5 \leq a + b\\
					s_{\alpha_1} s_{\eps_3 + \eps_4} s_{\eps_2 - \eps_1}s_{\alpha_1}\left(0, 0, 0, 0, -\lambda_5 + a + b, \lambda_6 - a - b - 2c, \ \lambda_7 + b + c, -\lambda_7 -b - c \right) \\
					=\left(0, 0, 0, 0, a, \lambda_5 + \lambda_6 - a - 2b - 2c, \lambda_5 + \lambda_7 + c, - \lambda_5 - \lambda_7 - c \right ), \ \lambda_5 < b
				\end{cases} \\
				&  = 	\begin{cases}
					\lambda  - \left(0, 0, 0, 0,  a + b,  a + b + 2c, - b - c, b + c \right), \ \lambda_5 > a + b \\
					\lambda  - \left(0, 0, 0, 0,  2 \lambda_5 - a- b,  a + b + 2c, - b - c, b + c \right), \ b \leq \lambda_5 \leq a + b\\
					\lambda  - \left(0, 0, 0, 0,  \lambda_5 - a, - \lambda_5 + a + 2b + 2c, - \lambda_5 - c, \lambda_5 + c \right), \ \lambda_5 < b.
				\end{cases}
			\end{align*}
			Then the strict Dirac inequality
			$$
			\|(\lambda - s_{a,b, c})^{+} + \rho  \|^2 > \| \lambda + \rho\|^2
			$$
			is equivalent to 
			$$
			\begin{cases}
				-2 \lambda_5 c - 2 \lambda_7(a + 2b + 3c) + 9a + 26b + 27 c < (a + b)^2 + 2(a + b)c +  2c^2 + (b + c)^2, \ \lambda_5 > a + b \\
				-2 \lambda_5 c - 2 \lambda_7(a + 2b + 3c) + 8 \lambda_5 +  a + 18b + 27 c < (a + b)^2 + 2(a + b)c + 2c^2 + (b + c)^2, \ b \leq \lambda_5 \leq a + b \\
				-2 \lambda_5 c - 2 \lambda_7(a + 2b + 3c) + 16 \lambda_5 +  a + 10b + 27 c < a^2 + 2a(b + c) + c^2 + 2(b + c)^2, \ \lambda_5 < b
			\end{cases}
			$$
			Let us assume that $\lambda_7 > 6$. Since $\lambda_5 \geq 0$, to prove the strict Dirac inequality it is enough to prove
			$$
			\begin{cases}
				-3a + 2b - 9c \leq (a + b)^2 + 2(a + b)c +  2c^2 + (b + c)^2, \ \lambda_5 > a + b \\
				-3a + 2b - 9c \leq  (a + b)^2 + 2(a + b)c + 2c^2 + (b + c)^2, \ b \leq \lambda_5 \leq a + b \\
				-11a + 2b - 9c \leq  a^2 + 2a(b + c) + c^2 + 2(b + c)^2, \ \lambda_5 < b
			\end{cases}.
			$$
			This is true for all $a, b, c \in \mathbb{N}_{0}, (a, b, c) \neq (0, 0, 0)$. So the strict Dirac inequality holds for any Schmid module $s_{a,b, c}$.

			
			\item If $4 < \lambda_7 < 6$ then 
			$$
			\|(\lambda - s_2)^{+} + \rho  \|^2 < \| \lambda + \rho\|^2.
			$$
			Since $s_1$ is the only Schmid module of strictly lower level than $s_2$ and for $s_1$ we have $ \lambda_7 > 4$, it follows that 
			$$
			\|(\lambda - s_1)^{+} + \rho  \|^2 > \| \lambda + \rho\|^2.
			$$
			
			\item If $\lambda_7 < 4$ than the basic Dirac inequality obviously fails since in Case 2 the basic Dirac inequality is equivalent to $\lambda_7 \geq 4$.
		\end{enumerate}
	\end{proof}
	
	\begin{lem}
		\label{third Schmid e7}
		Let $\lambda$ be a highest weight such that
		$$
		\|(\lambda - s_3)^{+} + \rho \|^2 > \| \lambda + \rho \|^2.
		$$
		Then 
		$$
		\|(\lambda' - s_3)^{+} + \rho \|^2 > \| \lambda' + \rho \|^2,
		$$
		where $\lambda' = (\lambda - s_3)^{+}$.
	\end{lem}
	
	\begin{proof}
		We have $$\lambda' = (\lambda - s_3)^{+} = \lambda - s_3 = (\lambda_1, \lambda_2, \lambda_3, \lambda_4, \lambda_5, \lambda_6 - 2, \lambda_7 + 1, -\lambda_7 - 1 ).$$ The strict Dirac inequality
		$$
		\|(\lambda' - s_3)^{+} + \rho \|^2 > \| \lambda' + \rho \|^2
		$$
		is equivalent to 
		$$
		\lambda_6' - \lambda_7' + 12 < 0
		$$
		and this is 
		equivalent to 
		$$
		\lambda_6 - \lambda_7 + 9 < 0,
		$$
		which is true since 
		$$
		\lambda_6 - \lambda_7 + 12 < 0.
		$$
	\end{proof}
	
	\begin{lem}\label{bs20}
		Let $\lambda$ be a highest weight such that $(\lambda_1, \lambda_2 ,\lambda_3, \lambda_4, \lambda_5) = (0, 0, 0, 0, 0)$ and
		$$
		\|(\lambda - s_2)^{+} + \rho \|^2 > \| \lambda + \rho \|^2.
		$$
		Then 
		$$
		\|(\lambda - s_{0, b, 0})^{+} + \rho \|^2 > \| \lambda + \rho \|^2 \quad \forall b \in \mathbb{N}.
		$$
	\end{lem}
	
	\begin{proof}
		We have $\lambda - s_{0, b, 0} = (0, 0, 0, 0, -b, \lambda_6 - b, \lambda_7 + b, - \lambda_7 - b)$. Now we have two cases.
		
		\textbf{Case 1:} $- \lambda_6 - 2 \lambda_7 - 2b \geq 0.$
		
		In this case 
		$$
		(\lambda - s_{0, b, 0})^{+} = (0, 0, 0, 0, b, \lambda_6 - b, \lambda_7 + b, - \lambda_7 - b) = \lambda - (0,0, 0, 0, -b, b, -b, b).
		$$
		The strict Dirac inequality 
		$$
		\|(\lambda - s_{0, b, 0})^{+} + \rho \|^2 > \| \lambda + \rho \|^2 
		$$
		is equivalent to 
		$$
		2 \left< \gamma \, , \, \lambda + \rho \right> < \| \gamma \|^2,
		$$
		where $\gamma = (0,0, 0, 0, -b, b, -b, b)$ and the last inequality is equivalent to
		$$
		\lambda_6 - 2 \lambda_7 + 18 < 2b.
		$$
		Since in this case  $- \lambda_6 - 2 \lambda_7 - 2b \geq 0$, then $\lambda$ is not in case $2.7.2.$. Therefore, $\lambda$ is in case $2. 7. 1.$. Since the strict Dirac inequality holds for the second basic Schmid module, we have $\lambda_6 - 2 \lambda_7 + 16 < 0$ and therefore $\lambda_6 - 2 \lambda_7 + 18 < 2 \leq 2b$. 
		
		\bigskip
		
		\textbf{Case 2:} $-\lambda_6 - 2 \lambda_7  - 2 b < 0.$
		
		Then 
		$$
		s_{\alpha_1}s_{\eps_3 + \eps_4}s_{\eps_2 - \eps_1}s_{\alpha_1}(0, 0, 0, 0, b, \lambda_6 - b, \lambda_7 + b, - \lambda_7 - b) = \left ( 0, 0, 0, 0, - \frac{\lambda_6 + 2 \lambda_7}{2}, \frac{\lambda_6}{2} - \lambda_7 - 2b, - \frac{\lambda_6}{2}, \frac{\lambda_6}{2}  \right ),
		$$
		so 
		$$
		(\lambda - s_{0, b, 0})^{+} =  \left( 0, 0, 0, 0, - \frac{\lambda_6 + 2 \lambda_7}{2}, \frac{\lambda_6}{2} - \lambda_7 - 2b, - \frac{\lambda_6}{2}, \frac{\lambda_6}{2} \right ) = \lambda - \gamma',
		$$
		where $\gamma' =  \left ( 0, 0, 0, 0, \frac{\lambda_6}{2} + \lambda_7, \frac{\lambda_6}{2} + \lambda_7 + 2b, \lambda_7 +  \frac{\lambda_6}{2}, - \lambda_7 - \frac{\lambda_6}{2} \right )$. The strict Dirac inequality
		$$
		\|(\lambda - s_{0, b, 0})^{+} + \rho \|^2 > \| \lambda + \rho \|^2 
		$$
		is equivalent to 
		\begin{equation}\label{middle}
			-2(\lambda_6 + 2 \lambda_7) < b \left (\lambda_7 - \frac{\lambda_6}{2} + b - 5 \right ).
		\end{equation}
		Since in this case we have $-\lambda_6 - 2 \lambda_7 < 2 b $, it is enough to prove
		$$
		4b \leq b \left ( \lambda_7 - \frac{\lambda_6}{2} + b - 5 \right ).
		$$
		The last inequality is equivalent to
		$$
		\lambda_6 - 2 \lambda_7 + 18 \leq 2b. 
		$$
		If $\lambda$ is in case $2.7.1$, then we have
		$$
		\lambda_6 - 2 \lambda_7 + 16 < 0,
		$$
		since we assumed that the strict Dirac inequality holds for the second basic Schmid module. Therefore 
		$$
		\lambda_6 - 2 \lambda_7 + 18 < 2 \leq 2b. 
		$$
		If $\lambda$ is in case $2.7.2$, then we have $\lambda_6 + 2 \lambda_7 = 0$, so inequality \eqref{middle} is equivalent to 
		$$
		\lambda_6 - 2 \lambda_7 < 2b - 10.
		$$ 
		Since we assumed that the strict Dirac inequality holds for the second basic Schmid module, we have
		$$
		\lambda_6 - 2 \lambda_7 < -8
		$$
		and therefore 
		$$
		\lambda_6 - 2 \lambda_7 < 2 - 10 \leq 2b - 10.
		$$
	\end{proof}

	\begin{lem}
		\label{second Schmid e7}
		Let $\lambda$ be a highest weight such that $(\lambda_1, \lambda_2 ,\lambda_3, \lambda_4, \lambda_5) \neq (0, 0, 0, 0, 0)$ and
		$$
		\|(\lambda - s_2)^{+} + \rho \|^2 > \| \lambda + \rho \|^2.
		$$
		Then 
		$$
		\|(\lambda' - s_2)^{+} + \rho \|^2 > \| \lambda' + \rho \|^2,
		$$
		where $\lambda' = (\lambda - s_2)^{+}$. If $\lambda_i' = 0$ for $i = 1, 2, 3, 4 ,5$, then 
		$$
		\|(\lambda' - s_{0, b, 0})^{+} + \rho \|^2 > \| \lambda' + \rho \|^2, \quad \forall b \in \mathbb{N}.
		$$
	\end{lem}
	\begin{proof}
		We have
		$$
		\lambda' = 
		\begin{cases}
			(\lambda_1, \lambda_2, \lambda_3, \lambda_4, \lambda_5 - 1, \lambda_6 - 1, \lambda_7 + 1, -\lambda_7 - 1), \ \lambda \text{ as in case 2.1} \\
			(\lambda_1, \lambda_2, \lambda_3, \ \lambda_5 - 1, \lambda_5, \lambda_6 - 1, \lambda_7 + 1, -\lambda_7 - 1), \ \lambda \text{ as in case 2.2} \\
			(\lambda_1, \lambda_2, \lambda_5 - 1, \lambda_5, \lambda_5, \lambda_6 - 1, \lambda_7 + 1, -\lambda_7 - 1), \ \lambda \text{ as in case 2.3} \\
			(\lambda_1, \lambda_5  - 1, \lambda_5, \lambda_5, \lambda_5, \lambda_6 - 1, \lambda_7 + 1, -\lambda_7 - 1), \ \lambda \text{ as in case 2.4} \\
			(\lambda_5 - 1, \lambda_5, \lambda_5, \lambda_5, \lambda_5, \lambda_6 - 1, \lambda_7 + 1, -\lambda_7 - 1), \ \lambda \text{ as in case 2.5.1.} \\
			(\lambda_5 - \frac{1}{2}, \lambda_5 - \frac{1}{2}, \lambda_5 - \frac{1}{2}, \lambda_5- \frac{1}{2}, \lambda_5 - \frac{1}{2}, \lambda_6 - \frac{3}{2}, \lambda_7 + \frac{1}{2}, -\lambda_7 - \frac{1}{2}), \ \lambda \text{ as in case 2.5.2.} \\
			(-\lambda_5 + 1, \lambda_5, \lambda_5, \lambda_5, \lambda_5, \lambda_6 - 1, \lambda_7 + 1, -\lambda_7 - 1), \ \lambda \text{ as in case 2.6}
		\end{cases}
		$$
		Therefore, 
		
		$$\lambda_5' + \lambda_6' - 2 \lambda_7' = 
		\begin{cases}
			\lambda_5 + \lambda_6 - 2 \lambda_7 - 4, \ \lambda \text{ as in case 2.1} \\ 
			\lambda_5 + \lambda_6 - 2 \lambda_7 - 3, \ \lambda \text{ as in case 2.2, 2.3, 2.4, 2.5.1, 2.5.2, 2.6}
		\end{cases}
		$$
		Since the strict Dirac inequality holds for the second basic Schmid module, we have 
		$$\lambda_5' + \lambda_6' - 2 \lambda_7' < 
		\begin{cases}
			-28, \ \lambda \text{ as in case 2.1} \\ 
			-26, \ \lambda \text{ as in case 2.2} \\ 
			-25, \ \lambda \text{ as in case 2.3} \\ 
			-24, \ \lambda \text{ as in case 2.4} \\
			-23, \ \lambda \text{ as in case 2.5.1 or 2.6} \\
			-22, \ \lambda \text{ as in case 2.5.2.}.   
		\end{cases}
		$$
		It is clear that
		$$
		\|(\lambda' - s_2)^{+} + \rho \|^2 > \| \lambda' + \rho \|^2
		$$
		if $\lambda$ is in one of the cases $2.1, 2.2, 2.3$ or $2.4$. If $\lambda$ is as in case $2.5.1$ or $2.6$, then $\lambda'$ is not as in case $2.1$ and therefore
		$$
		\|(\lambda' - s_2)^{+} + \rho \|^2 > \| \lambda' + \rho \|^2.
		$$
		If $\lambda$ is as in case $2.5.2$, then $\lambda'$ is not as in case $2.1$ or $2.2$ 
		and therefore
		$$
		\|(\lambda' - s_2)^{+} + \rho \|^2 > \| \lambda' + \rho \|^2.
		$$
		
		So the strict Dirac inequality holds for the second basic Schmid module for the weight $\lambda'$.
		
		If $\lambda_5' =  \lambda_4' = \lambda_3' = \lambda_2' = \lambda_1'= 0$, then  it follows from lemma \ref{bs20} that
		$$
		\|(\lambda' - s_{0,b, 0})^{+} + \rho\|^2 > \|\lambda' + \rho \|^2  \quad \forall b \in \mathbb{N}.
		$$ 
	\end{proof}

	\begin{lem}
		\label{first Schmid e7}
		Let $\lambda$ be a highest weight such that $\lambda$ is as in case $3$ and
		$$
		\|(\lambda - s_1)^{+} + \rho \|^2 > \| \lambda + \rho \|^2.
		$$
		Then 
		$$
		\|(\lambda' - s_1)^{+} + \rho \|^2 > \| \lambda' + \rho \|^2,
		$$
		where $\lambda' = (\lambda - s_1)^{+}$. If $\lambda'$ is as in Case 1 or Case 2, then 
		$$
		\|(\lambda' - s_{a, b, c})^{+} + \rho \|^2 > \| \lambda' + \rho \|^2, \quad \forall a, b, c \in \mathbb{N}_{0}, \quad a + b + c \neq 0.
		$$
	\end{lem}
	\begin{proof}
		We have
		$$
		\lambda' = 
		\begin{cases}
			(\lambda_1, \lambda_2, \lambda_3, \lambda_4, \lambda_5, \lambda_6, \lambda_7 + 1, -\lambda_7 - 1), \ \lambda \text{ as in case 1.1} \\
			(\lambda_1 + \frac{1}{2}, \lambda_2 - \frac{1}{2}, \lambda_3 - \frac{1}{2}, \lambda_4 - \frac{1}{2}, \lambda_5 - \frac{1}{2}, \lambda_6 - \frac{1}{2}, \lambda_7 + \frac{1}{2}, -\lambda_7 - \frac{1}{2}), \ \lambda \text{ as in case 1.2.1} \\
			(\lambda_2 - \frac{1}{2}, \lambda_2 + \frac{1}{2}, \lambda_3 - \frac{1}{2}, \lambda_4 - \frac{1}{2}, \lambda_5 - \frac{1}{2}, \lambda_6 - \frac{1}{2}, \lambda_7 + \frac{1}{2}, -\lambda_7 - \frac{1}{2}), \ \lambda \text{ as in case 1.2.2} \\
			(\lambda_3 - \frac{1}{2}, \lambda_3 - \frac{1}{2}, \lambda_3 + \frac{1}{2}, \lambda_4 - \frac{1}{2}, \lambda_5 - \frac{1}{2}, \lambda_6 - \frac{1}{2}, \lambda_7 + \frac{1}{2}, -\lambda_7 - \frac{1}{2}), \ \lambda \text{ as in case 1.2.3} \\
			( \frac{1}{2}, \frac{1}{2},  \frac{1}{2}, \lambda_4 - \frac{1}{2}, \lambda_5 - \frac{1}{2}, \lambda_6 - \frac{1}{2}, \lambda_7 + \frac{1}{2}, -\lambda_7 - \frac{1}{2}), \ \lambda \text{ as in case 1.2.4} \\
			(\lambda_4 - \frac{1}{2}, \lambda_4 - \frac{1}{2}, \lambda_4 - \frac{1}{2}, \lambda_4 + \frac{1}{2}, \lambda_5 - \frac{1}{2}, \lambda_6 - \frac{1}{2}, \lambda_7 + \frac{1}{2}, -\lambda_7 - \frac{1}{2}), \ \lambda \text{ as in case 1.2.5} \\
			(\lambda_5 - \frac{1}{2}, \lambda_5 - \frac{1}{2}, \lambda_5 - \frac{1}{2}, \lambda_5 - \frac{1}{2}, \lambda_5 + \frac{1}{2}, \lambda_6 - \frac{1}{2}, \lambda_7 + \frac{1}{2}, -\lambda_7 - \frac{1}{2}), \ \lambda \text{ as in case 1.2.7}
		\end{cases}
		$$
		Since 	
		$$
		\|(\lambda - s_1)^{+} + \rho \|^2 > \| \lambda + \rho \|^2,
		$$
		it follows that
		\begin{equation}\label{cases lambda7'}
			\lambda_7' > 
			\begin{cases}
				9, \ \lambda \text{ as in case 1.1} \\
				8, \ \lambda \text{ as in case 1.2.1} \\
				7  + \frac{1}{2}, \ \lambda \text{ as in case 1.2.2} \\
				7, \ \lambda \text{ as in case 1.2.3} \\
				6 + \frac{1}{2}, \ \lambda \text{ as in case 1.2.4} \\
				6 + \frac{1}{2}, \ \lambda \text{ as in case 1.2.5} \\
				6, \ \lambda \text{ as in case 1.2.7}
			\end{cases}
		\end{equation}
		It is clear that 
		$$
		\|(\lambda' - s_1)^{+} + \rho \|^2 > \| \lambda' + \rho \|^2
		$$
		if $\lambda$ is as in case $1.1$ or $1.2.1$. If $\lambda$ is as in case $1.2.2$, then $\lambda'$ is not as in case $1.1$. Also if $\lambda$ is as in case $1.2.3$, then $\lambda'$ is neither as in case $1.1$ nor as in case $1.2.1$. If $\lambda$ is as in case $1.2.4$ or $1.2.5$, then $\lambda'$ is not in any of the cases $1.1, 1.2.1, 1.2.2$. If $\lambda$ is in case $1.2.7$, then $\lambda'$ is in none of the cases $1.1, 1.2.1, 1.2.2, 1.2.3$. It follows from (\ref{cases lambda7'}) that
		$$
		\|(\lambda' - s_1)^{+} + \rho \|^2 > \| \lambda' + \rho \|^2.
		$$
		Furthermore,  it follows from (\ref{cases lambda7'}) that if $\lambda$ is as in Case 3, then $\lambda_7' > 6$. Therefore, it follows  from the proof of  theorem \ref{first five zero e7} and the proof of  theorem \ref{first four zero e7}  that if  $\lambda'$ is as in Case $1$ (case $1.2.8.$) or as in Case $2$ (case $1.2.6.$), then 
		$$
		\|(\lambda' - s_{a, b, c})^{+} + \rho \|^2 > \| \lambda' + \rho \|^2 \quad \forall a, b, c \in \mathbb{N}_{0},\  a + b + c \neq 0.
		$$
	\end{proof}

	\begin{thm} (Case 3) Let $\lambda$ be the highest weight as in Case 3 such that the strict basic Dirac inequality holds. Then 
		$$
		\|(\lambda - s_{a,b,c})^{+} + \rho  \|^2 > \| \lambda + \rho\|^2, \ a, b, c \in \mathbb{N}_{0}, \ (a,b,c) \neq (0, 0, 0)
		$$
	\end{thm}
	
	
	\begin{proof}
		Let $\lambda$ be as in Case 3, and let us assume that the strict basic Dirac inequality holds. First we will prove that in this case we have
		\begin{equation}\label{be7}
			\|(\lambda - s_{0,b, 0})^{+} + \rho\|^2 - \|\lambda + \rho \|^2 > 0 \quad \forall b \in \mathbb{N}.
		\end{equation}
		Let us denote $\lambda' = (\lambda - s_2)^{+}$. We have already proved that if $\lambda$ is in Case 3 and  the strict basic Dirac inequality holds, then the strict Dirac inequality also holds for $s_2$. So we have 
		\begin{equation}\label{lambda's2}
			\| \lambda' + \rho\|^2 > \| \lambda + \rho\|^2.
		\end{equation}
		If $\lambda_1 = \lambda_2 = \lambda_3 = \lambda_4 = \lambda_5 = 0$ then \eqref{be7} obviously follows from Lemma \ref{bs20}. Let us assume that $(\lambda_1, \lambda_2,  \lambda_3, \lambda_4, \lambda_5) \neq (0, 0, 0, 0, 0)$. From Lemma \ref{second Schmid e7} it follows that 
		\begin{equation}\label{lambda''s2}
			\| (\lambda' - s_2)^{+} + \rho \|^2 > \|\lambda' + \rho \|^2.
		\end{equation}
		
		Let us assume that $b > 1$ (\eqref{be7} obviously holds for $b = 1$ since $s_{0,b,0} = s_2$). Let $w\in W_{\mathfrak{k}}$ be such that $\lambda - s_2 = w(\lambda - s_2)^{+}$.  From Lemma \ref{gen prv} we have
		\begin{align*}
			\|(\lambda - s_{0,b,0})^{+} + \rho\|^2 & = \|(\lambda - s_2  - s_{0,b - 1, 0})^{+} + \rho\|^2 
			= \| (w(\lambda - s_2)^{+} -  s_{0, b - 1, 0})^{+} + \rho \|^2 \\
			& \geq \|((\lambda - s_2)^{+} - s_{0, b - 1, 0})^{+} + \rho\|^2 =  \|(\lambda' - s_{0, b- 1, 0})^{+} + \rho\|^2.
		\end{align*}
		It follows from the last inequality and  \eqref{lambda's2} that
		\begin{equation}\label{e7b}
			\|(\lambda - s_{0,b, 0})^{+} + \rho\|^2 - \| \lambda + \rho\|^2 > \|(\lambda' - s_{0,b - 1, 0})^{+} + \rho\|^2 - \| \lambda' + \rho\|^2.
		\end{equation}
		
		If $\lambda_i' = 0$ for $i = 1, 2, 3, 4 ,5$, then  it follows from Lemma \ref{bs20} and \eqref{lambda''s2}
		$$
		\|(\lambda' - s_{0, b-1, 0})^{+} + \rho \|^2 > \| \lambda' + \rho \|^2.
		$$
		and by \eqref{e7b}
		$$
		\|(\lambda - s_{0, b, 0})^{+} + \rho \|^2 - \| \lambda + \rho \|^2 > 0.
		$$
		
		If $(\lambda_1', \lambda_2', \lambda_3', \lambda_4', \lambda_5' ) \neq (0,0,0,0,0)$ and $b = 2$, then \eqref{be7} follows from \eqref{e7b} and \eqref{lambda''s2}. 
		
		If $(\lambda_1', \lambda_2', \lambda_3', \lambda_4', \lambda_5' ) \neq (0,0,0,0,0)$ and $b > 2$, then we have
		$$
		\|(\lambda' - s_{0,b-1,0})^{+} + \rho\|^2 - \| \lambda' + \rho\|^2 > \|(\lambda'' - s_{0,b-2,0})^{+} + \rho\|^2 - \| \lambda'' + \rho\|^2, 
		$$
		where $\lambda'' = (\lambda' - s_2)^{+}$. By induction, it follows that
		$$
		\|(\lambda - s_{0,b,0})^{+} + \rho\|^2 - \| \lambda + \rho\|^2 > 0 \quad  \forall b \in \mathbb{N}.
		$$

		Now we will prove that 	
		\begin{equation}\label{bc7}
			\|(\lambda - s_{0,b, c})^{+} + \rho\|^2 - \|\lambda + \rho \|^2 > 0 \quad \forall b,c \in \mathbb{N}_{0}, \ b + c \neq 0.
		\end{equation}
		Let us denote $\lambda''' = (\lambda - s_3)^{+}$. Since $\lambda$ is in Case 3, it is easy to check that $\lambda'''$ is also in Case 3. We have already proved that if $\lambda$ is in Case 3 and  the strict basic Dirac inequality holds, then the strict Dirac inequality also holds for $s_3$. So we have 
		$$
		\| \lambda''' + \rho\|^2 > \| \lambda + \rho\|^2.
		$$
		Let $w'\in W_k$ be such that $\lambda - s_3 = w'(\lambda - s_3)^{+}$.  From Lemma \eqref{gen prv} we have
		\begin{align*}
			\|(\lambda - s_{0,b,c})^{+} + \rho\|^2 & = \|(\lambda - s_3  - s_{0,b, c - 1})^{+} + \rho\|^2 
			= \| (w'(\lambda - s_3)^{+} -  s_{0, b, c - 1})^{+} + \rho \|^2 \\
			& \geq \|((\lambda - s_3)^{+} - s_{0, b, c - 1})^{+} + \rho\|^2 =  \|(\lambda''' - s_{0, b, c - 1})^{+} + \rho\|^2,
		\end{align*}
		if $c > 1$.
		From the last two inequalities it follows that
		\begin{equation}\label{e7bc}
			\|(\lambda - s_{0,b, c})^{+} + \rho\|^2 - \| \lambda + \rho\|^2 > \|(\lambda''' - s_{0,b, c - 1})^{+} + \rho\|^2 - \| \lambda''' + \rho\|^2 \quad \forall b, c \in \mathbb{N}_{0}, b + c \neq 0
		\end{equation}
		
		Now \eqref{bc7} follows from Lemma \ref{third Schmid e7}, \eqref{e7bc} and \eqref{be7} by induction on $c$. 
		
		\bigskip
		
		Now we will prove that if $\lambda$ is in Case 3, and the strict basic Dirac inequality holds, then 
		$$
		\|(\lambda - s_{a,b,c})^{+} + \rho\|^2 - \|\lambda + \rho \|^2 > 0 \quad \forall a, b,c \in \mathbb{N}_{0}, (a, b, c) \neq (0, 0, 0).
		$$ 
		Lat us assume that $a > 0$ (if $a = 0$, the last inequality is exactly \eqref{bc7}).
		Let us denote $\tilde{\lambda} = (\lambda - s_1)^{+}$. We have 
		$$
		\| \tilde{\lambda} + \rho\|^2 > \| \lambda + \rho\|^2.
		$$
		Let $\tilde{w}\in W_k$ be such that $\lambda - s_1 = \tilde{w}(\lambda - s_1)^{+}$.  From Corollary 2.9 we have
		\begin{align*}
			\|(\lambda - s_{a,b, c})^{+} + \rho\|^2 & = \|(\lambda - s_1  - s_{a - 1, b, c})^{+} + \rho\|^2 
			= \| (\tilde{w}(\lambda - s_1)^{+} -  s_{a - 1, b, c})^{+} + \rho \|^2 \\
			& \geq \|((\lambda - s_1)^{+} - s_{a-1,b, c})^{+} + \rho\|^2 =  \|(\tilde{\lambda} - s_{a - 1,b, c})^{+} + \rho\|^2.
		\end{align*}
		From the last two inequalities it follows that
		\begin{equation}\label{e7abc}
			\|(\lambda - s_{a,b,c})^{+} + \rho\|^2 - \| \lambda + \rho\|^2 > \|(\tilde{\lambda} - s_{a - 1,b, c})^{+} + \rho\|^2 - \| \tilde{\lambda} + \rho\|^2.
		\end{equation}
		If $\tilde{\lambda}$ is in Case 1 or Case 2, then
		it follows from lemma \ref{first Schmid e7} that
		$$
		\|(\tilde{\lambda} - s_{a - 1, b, c})^{+} + \rho \|^2 > \| \tilde{\lambda} + \rho \|^2,
		$$
		and from \eqref{e7abc} it follows that
		$$
		\|(\lambda - s_{a, b, c})^{+} + \rho \|^2 - \| \lambda + \rho \|^2 > 0 \quad \forall a, b, c \in \mathbb{N}_{0}, a + b + c \neq 0.
		$$
		If $\tilde{\lambda}$ is in Case 3 and $a > 1$ then it follows from Lemma \ref{first Schmid e7} and \eqref{e7abc} that 
		$$
		\|(\tilde{\lambda} - s_{a - 1,b, c})^{+} + \rho\|^2 - \| \tilde{\lambda} + \rho\|^2 > \|(\bar{\lambda} - s_{a - 2,b, c})^{+} + \rho\|^2 - \| \bar{\lambda} + \rho\|^2, 
		$$
		where $\bar{\lambda} = (\tilde{\lambda} - s_1)^{+}$. By induction on $a$ and by (\ref{bc7}), it follows
		$$
		\|(\lambda - s_{a,b, c})^{+} + \rho\|^2 - \| \lambda + \rho\|^2 > 0 \quad  \forall a, b, c \in \mathbb{N}_{0}, (a, b, c) \neq (0,0, 0).
		$$
	\end{proof}

\end{document}